\documentclass{amsart}

\usepackage[T1]{fontenc}
\usepackage{amsmath}
\usepackage{amsthm}
\usepackage{amsfonts}
\usepackage{amssymb}
\usepackage{enumerate}
\usepackage{mathrsfs}
\usepackage{graphicx}
\usepackage{hyperref}
\usepackage{enumitem}
\usepackage[labelfont=rm]{caption}
\usepackage{subcaption}
\usepackage{tikz}
\usepackage{defs}
\newif\iftikziii
\tikziiitrue

\iftikziii
\usetikzlibrary{matrix,calc,arrows.meta,knots,intersections,decorations.markings,cd}
\else
\usetikzlibrary{matrix,calc,knots,intersections,decorations.markings}
\fi

\def\wt{\widetilde}
\newcommand{\floor}[1]{\left \lfloor #1 \right \rfloor}

\DeclareMathOperator{\ord}{ord}

\title{Coverings of rational ruled normal surfaces}
\author[E. Artal]{Enrique Artal Bartolo}

\author[J.I. Cogolludo]{Jos{\'e} Ignacio Cogolludo-Agust{\'i}n}
\address{Departamento de Matem\'aticas, IUMA\\ 
Universidad de Zaragoza\\ 
C.~Pedro Cerbuna 12\\ 
50009 Zaragoza, Spain} 
\email{artal@unizar.es,jicogo@unizar.es} 

\author[J.~Mart\'{i}n-Morales]{Jorge Mart\'{i}n-Morales}
\address{Centro Universitario de la Defensa-IUMA \\
Academia General Militar \\
Ctra.~de Huesca s/n. 50090, Zaragoza, Spain}
\email{jorge@unizar.es}
\urladdr{\url{http://cud.unizar.es/martin}}

\thanks{Partially supported by
MTM2016-76868-C2-2-P}  



\begin{document}

\begin{abstract}
In this work we use arithmetic, geometric, and combinatorial techniques to compute
the cohomology of Weil divisors of a special class of normal surfaces, 
the so-called rational ruled toric surfaces. These computations are used to study
the topology of cyclic coverings of such surfaces ramified along
$\mathbb{Q}$-normal crossing divisors.
\end{abstract}

\maketitle

\section*{Introduction}

The main purpose of this paper is the study of a special type of rational normal surfaces generalizing 
Hirzebruch surfaces, that carry a rational fibration. These are called rational ruled toric 
surfaces (see section~\ref{sec:RRTsurfaces} for a precise definition).

The first goal is to study the Picard group of these rational surfaces. Note that, while
the Picard group of smooth rational surfaces is finitely generated and torsion free,
this last property may be lost in some cases for toric ruled surfaces.

The second goal is to study the cohomology of the sheaves $\cO_S(D)$ for any Weil divisor $D$ of~$S$.

The final purpose is to apply this in order to calculate the first Betti number of cyclic covers of~$S$
with a prescribed ramification divisor~$D$.

This will be done in several steps:
\begin{itemize}
 \item 
Compute the Euler characteristic of $\cO_S(D)$ (Lemma~\ref{lemma:chigen}).
Using the Riemann-Roch formula for normal surfaces (see Blache~\cite{Blache-RiemannRoch}) one can
obtain the Euler characteristic of $\cO_S(D)$ in terms of intersection numbers and a correction 
term coming from the singular locus $\Sing(S)$.
 \item 
Compute $H^0(S,\cO_S(D))$ (sections~\ref{sec:RRTsurfaces-specialtype} and \ref{sec:RRTsurfaces-generaltype}). 
Using the contributions of Sakai~\cite{Sakai84} and the theory of weighted blow-ups, 
we can express these groups as subspaces of $0$-cohomology groups for line bundles in~$\bp^2$, 
defined by imposing weighted-multiplicity conditions at some points. These computations can be 
translated into the counting of integer points inside rational polygons, that is, one whose 
vertices are rational points (as in~\ref{sec:RRTsurfaces-specialtype}\eqref{eq:h0D1}).
 \item 
Calculate $H^2(S,\cO_S(D))$ (sections~\ref{sec:RRTsurfaces-specialtype} and \ref{sec:RRTsurfaces-generaltype}). 
Using Serre duality, $H^2(S,\cO_S(D))$ will be obtained as the dual 
of $H^0(S,\cO_S(K_S-D))$, where $K_S$ is a canonical divisor.
 \item
Calculate $b_1(\tilde S)$ of a cyclic cover of $S$ (section~\ref{sec:esnault}). 
This is done using Esnault-Viehweg type of formulas adapted to the quotient singularity case.
\end{itemize}

With these data, complete computations will be carried out in section~\ref{sec:RRTsurfaces-specialtype}
for biruled surfaces (admitting sections with vanishing self-intersection) 
and in section~\ref{sec:RRTsurfaces-generaltype} for uniruled surfaces.

The main result of our paper in this direction is Theorem~\ref{thm:main} for biruled surfaces, 
which states that given any divisor, at most one cohomology group does not vanish.
This is still true for uniruled surfaces as shown in Theorems~\ref{thm:main2} and~\ref{thm:h02} 
when the divisor is located in some special \emph{regions}, as it is well known to be the case for 
smooth ruled surfaces different from~$\bp^1\times\bp^1$.

As mentioned above, the main application of this computation is the study of Betti numbers of the
$n$-cyclic coverings of such a surface~$S$, which are ramified at $\QQ$-normal crossing
divisors of~$S$. This is done by extending classical results of Esnault-Viehweg to the quotient singularity case,
where this Betti number computation of the cyclic cover is reduced to the computation
of $1$-cohomology groups of some divisors associated with the ramification divisor $D$ and
to some divisor~$H$ such that $D$ and $dH$ are linearly equivalent. This information is finer than the 
one required in the smooth case. The main subtlety in the normal case comes from the fact that the 
codimension~$1$ ramification locus does not determine the codimension~$2$ ramification.

This is the first step of a more ambitious project, including the extension of
Esnault-Viehweg results and their application to more general ruled surfaces
having quotient singular points at two disjoint sections. The reason
behind these computations is the complete study of the monodromy
of L{\^e}-Yomdine singularities where these computations
are the key missing points.

The paper is organized as follows. In section~\ref{sec:prelim}, we recall basic tools in this theory such as 
weighted blow-ups, Blache-Sakai theory of normal surfaces, the behavior of Picard groups by birational 
morphisms, and the Riemann-Roch formula with its correction terms for surfaces with quotient singular points. 
Section~\ref{sec:RRTsurfaces} is devoted to analyzing the properties of toric ruled surfaces, including the
description of their Picard groups.
In sections~\ref{sec:RRTsurfaces-specialtype} and \ref{sec:RRTsurfaces-generaltype}, the cohomology of Weil 
divisors of toric biruled and uniruled surfaces is studied by using combinatorial and arithmetical tools.
The extension of Esnault-Viehweg results is presented in section~\ref{sec:esnault},
which ends with some illustrative examples.

During the preparation of this work, we have combined geometric and algebraic techniques
in a spirit close to one we believe is characteristic of Antonio Campillo's to whom we dedicate this work.

\section{Preliminaries}
\label{sec:prelim}

\subsection{Quotient singularities and weighted blow-ups}\label{sec:qs}

We will denote
by $\mu_d$ the group of $d$-roots of unity in~$\bc^*$.

\begin{dfn}\label{def:csp}
A complex algebraic surface $S$ has a \emph{cyclic singular point of type $\frac{1}{d}(a,b)$} at $0\in S$ if 
$(S,0)$ is locally is isomorphic to the germ $(\bc^2/\mu_d,0)$ obtained by the action 
$\zeta\cdot(x,y):=(\zeta^a x,\zeta^b y)$ on~$\bc^2$. 
\end{dfn}

From now on, $\frac{1}{d}(a,b)$ will denote the germ of
the surface at such a singular point.

\begin{obs}
Even though we do not require a priori any condition on the integers $d,a,b$ other than $d>0$,
there is no actual restriction in assuming that $u:=\gcd(d,a,b)$ is~$1$
since $\frac{1}{d}(a,b)=\frac{u}{d}(\frac{a}{u},\frac{b}{u})$. If $u=1$, we can
also assume that $v:=\gcd(a,b)$ is also~$1$ since $\frac{1}{d}(a,b)=\frac{1}{d}(\frac{a}{v},\frac{b}{v})$. 
Finally, note that if $w:=\gcd(d,a)$, then the map $(x,y)\mapsto(x,y^w)$ induces an isomorphism
$\frac{1}{d}(a,b)\to \frac{w}{d}(\frac{a}{w},b)$. Hence, we may also assume
$\gcd(d,a)=\gcd(d,b)=1$.
\end{obs}

\begin{ntc}
\label{ntc:cyclic}
Sometimes it is useful to keep track of the divisors appearing in the definition of $\frac{1}{d}(a,b)$, 
and write $\frac{1}{d}(a_A,b_B)$, where $A,B$ are Weil divisors whose \emph{local equation in $\bc^2$} 
are $x=0$ and $y=0$, respectively as in Definition~\ref{def:csp}. 
\end{ntc}

\begin{ejm}
The first examples are weighted projective spaces. Let us start
with \emph{weighted} projective lines. Let $\omega:=(p,q)$ be coprime positive integers. We consider in 
$\bc^2\setminus\{0\}$ the equivalence relation given by
$(x,y)\sim(t^p x,t^q y)$, for $t\in\bc^*$. Its quotient
is the $\omega$-weighted projective line $\bp^1_\omega$ and its
elements are denoted by $[x:y]_\omega$. This space can be understood
using \emph{charts}. Let
\[
\begin{tikzcd}[column sep=1em,row sep=.1em]
\bc_x\arrow[rr,"\sigma_x"]&&\bp^1_\omega&&\bc_y\arrow[rr,"\sigma_y"]&&\bp^1_\omega\\
x\arrow[rr,mapsto]&&{[x:1]}_\omega&&y\arrow[rr,mapsto]&&{[1:y]}_\omega
\end{tikzcd}
\]
The images of these maps cover $\bp^1_\omega$ but they are not actually charts,
and they define charts if we consider the quotients $\bc_x/\mu_q$ and $\bc_y/\mu_p$.
Note that $\bc_x/\mu_q\to\bc$, $[x]\mapsto x^q$, and $\bc_y/\mu_p\to\bc$,
$[y]\mapsto y^p$, are isomorphisms, and  $\bp^1_\omega$ is actually $\bp^1$.
\end{ejm}

\subsubsection{Weighted blow-ups of smooth points}
\label{sec:blowup}

Let $P\in S$ be a smooth point, and let $A,B$ two divisors which are normal crossing
at~$P$; let us suppose that we have local coordinates such that $A:x=0$ and $B:y=0$.
Let $\omega:=(p,q)$ be coprime positive integers. The $(p,q)$-weighted blowing-up
of $P$ (relative to $A,B$) is a map $\rho:S_{P,\omega}\to S$, which is an isomorphism outside~$P$ and the model near~$P$ is as follows:
\begin{equation}\label{eq:pq}
\begin{tikzcd}[column sep=.5em,row sep=.1em]
\hat{\bc}^2_\omega:=\{((x,y),[u:v])\in\bc^ 2\times\bp^1_\omega\mid x^q v^p=y^p u^q
\}\arrow[rrrr,"\rho"]&&&&\bc^2\\
((x,y),[u:v])\arrow[rrrr,mapsto]&&&&(x,y)
\end{tikzcd}
\end{equation}
Let $E:=\rho^{-1}(0)$; it is abstractly isomorphic to $\bp^1_\omega\cong\bp^1$. Note that $S_{P,\omega}$ may have
singular points if either $p$ or~$q$ are greater than~$1$. Let us study for that $\hat{\bc}^2_\omega$.
It will be covered by two maps
\begin{equation}\label{eq:cartaspq}
\begin{tikzcd}[column sep=1em,row sep=.1em]
\bc^2_x\arrow[rr,"\tau_x"]&&\hat\bc^2_\omega&&\bc^2_y\arrow[rr,"\tau_y"]&&\hat{\bc}^2_\omega\\
(x,y)\arrow[rr,mapsto]&&
((x y^p,y^q),{[x:1]}_\omega)
&&(x,y)\arrow[rr,mapsto]&&((x^p,x^q y),{[1:y]}_\omega)
\end{tikzcd}
\end{equation}
If we want these maps to be isomorphisms onto the image, we need to replace $\bc^2_x$ by $\frac{1}{q}(p_A,-1_E)$
and $\bc^2_y$ by $\frac{1}{p}(-1_E,q_B)$. Using the properties shown in \cite{amo:jos}, we have
that $E^2=-\frac{1}{pq}$. Let us keep the same notation for the strict transforms. Then,
\[
A\cdot E=\frac{1}{q},\quad B\cdot E=\frac{1}{p},\quad (A\cdot A)_{S_{P,\omega}}=(A\cdot A)_S-\frac{p}{q}
,\quad (B\cdot B)_{S_{P,\omega}}=(B\cdot B)_S-\frac{q}{p}.
\]

\subsubsection{Weighted blow-ups of singular points: special case}
\label{sec:blowup-pq}

Quotient singular points do also admit weighted blow-ups. We present the two examples which will be used later.

Let $P\in S$ be a singular point of type $\frac{1}{d}(p_A,q_B)$, and let $A,B$ two divisors which are $\bq$-normal crossing
at~$P$, with \emph{local coordinates} such that $A:x=0$ and $B:y=0$, with $\gcd(p,q)=\gcd(d,p)=\gcd(d,q)=1$. Denote $\omega:=(p,q)$; the $(p,q)$-weighted blowing-up
of $P$ (relative to $A,B$) is a map $\rho:S_{P,\omega}\to S$, which is an isomorphism outside~$P$ and the model near~$P$ is a quotient of~\eqref{eq:pq}. We cover this space
with two maps as in~\eqref{eq:cartaspq}. Let us consider the origin of~$\bc^2_x$; besides the action of
$\mu_q$, the following commutative diagram holds:
\[
\begin{tikzcd}[column sep=3em,ampersand replacement=\&]
(x,y)\arrow[rrr,mapsto]\arrow[dddd,mapsto]\&\&\&(x y^p,y^q)\arrow[dddd,mapsto]\\[-20pt]
\&\bc^2_x\arrow[r,"\tau_x"]\arrow[dd]\&\bc^ 2\arrow[dd]\&\\[-10mm]
\&\&\&\&[-2cm]{
\Longrightarrow \begin{pmatrix}[c|cc]
q&p&-1\\
d&0&1
\end{pmatrix}
\cong \frac{1}{q}(p,-d)
}
\\[-10mm]
\&\bc^2_x\arrow[r,"\tau_x"]\&\bc^2\&\\[-20pt]
( x,\zeta_{d} y)\arrow[rrr,mapsto]\&\&\&
(x (\zeta_d y)^p,(\zeta_d y)^q)
\end{tikzcd}
\]
where $
\left(\begin{smallarray}{c|cc}
q&p&-1\\
d&0&1
\end{smallarray}
\right)
$
represents the quotient of $\bc^2$ by the group $\mu_q\times\mu_d$ indicated
by the right hand side of the matrix as in Definition~\ref{def:csp}.
Hence, in $E:=\rho^{-1}(0)\cong\bp^1$
we have two points of type $\frac{1}{q}(p_A,-d_E)$
and $\frac{1}{p}(-d_E,q_B)$. The following holds:
\begin{equation}\label{eq:intersections_pq}
E^2=-\frac{d}{pq}\quad A\cdot E=\frac{1}{q},\quad B\cdot E=\frac{1}{p},\quad 
(A\cdot A)_{S_{P,\omega}}=(A\cdot A)_S-\frac{p}{d q}
,\quad (B\cdot B)_{S_{P,\omega}}=(B\cdot B)_S-\frac{q}{d p}.
\end{equation}

\begin{ejm}\label{ejm:weighted_Cremona}
Let $\omega:=(p,q,r)$ be a triple of
pairwise coprime positive integer numbers. Let $\bp^2_\omega$ be the projective plane 
associated to~$\omega$, i.e.. the quotient of $\bc^3\setminus\{0\}$ by the action
$t\cdot(x,y,z):=(t^p x,t^q y, t^r z)$, whose elements will be denoted by
$[x:y:z]_\omega$. Let $X\subset\bp^2_\omega$ be the curve defined by $x=0$, and
define in the same way the curves $Y,Z$. 

Consider the map $\rho:\bp^2\to\bp^2_\omega$ given by $\rho([x:y:z]):=[x^p:y^q:z^r]_\omega$
(of degree~$pqr$).
Note that $X,Y,Z$ are Weil divisors and
\[
X\cdot Y=\frac{1}{r},\quad Y\cdot Z=\frac{1}{p},\quad X\cdot Z=\frac{1}{q}
,\quad X^2=\frac{p}{q r},\quad Y^2=\frac{q}{p r},\quad X^2=\frac{r}{p q}.
\]
Let us give another way to construct this weighted projective plane.
\begin{figure}[ht]
\begin{center}
\begin{tikzpicture}[scale=.6,
vertice/.style={draw,circle,fill,minimum size=0.2cm,inner 
sep=0}
]
\coordinate (R) at (0,2);
\coordinate  (P) at (-2,0) ;
\coordinate (Q) at (2,0);

\draw ($1.1*(P)-.1*(Q)$) -- node[below] {$Z$} ($1.1*(Q)-.1*(P)$);
\draw ($1.1*(P)-.1*(R)$) -- node[above=5pt] {$X$} ($1.1*(R)-.1*(P)$);
\draw ($1.1*(R)-.1*(Q)$) --  node[above=5pt, pos=.6] {$Y$} ($1.1*(Q)-.1*(R)$);
\node[above=3pt] at (R) {$(1,1)$};
\node[below left=3pt] at (P) {$(p,\alpha)$};
\node[below right=3pt] at (Q) {$(q,\beta)$};

\begin{scope}[xshift=8cm,yshift=1cm]
\coordinate (P1) at (-1,2);
\coordinate (P2) at (1,2);
\coordinate  (P3) at (-2,0) ;
\coordinate (P4) at (2,0);
\coordinate (P5) at (-1,-2);
\coordinate (P6) at (1,-2);
\draw ($1.1*(P1)-.1*(P2)$) -- node[below] {${E_Z}_{(-1)}$} node [below] {} ($1.1*(P2)-.1*(P1)$);
\draw ($1.1*(P2)-.1*(P4)$) -- node[right] {$Y_{(-\frac{q}{\beta})}$} node[left] {} ($1.1*(P4)-.1*(P2)$);
\draw ($1.1*(P6)-.1*(P4)$) -- node[right] {${E_X}_{(-\frac{1}{q\beta})}$} node[left] {} ($1.1*(P4)-.1*(P6)$);
\draw ($1.1*(P6)-.1*(P5)$) -- node[above] {$Z_{(-\frac{r}{pq})}$} node[above] {}  ($1.1*(P5)-.1*(P6)$);
\draw ($1.1*(P3)-.1*(P5)$) -- node[left] {${E_Y}_{(-\frac{1}{p\alpha})}$} node[right] {} ($1.1*(P5)-.1*(P3)$);
\draw ($1.1*(P3)-.1*(P1)$) -- node[left] {$X_{(-\frac{p}{\alpha})}$} node[right] {} ($1.1*(P1)-.1*(P3)$);
\node[vertice] at (P6) {};
\node[below right] at (P6) {$\frac{1}{q}(-1,\beta)$};
\node[vertice] at (P5) {};
\node[below left] at (P5) {$\frac{1}{p}(-1,\alpha)$};
\node[vertice] at (P3) {};
\node[left] at (P3) {$\frac{1}{\alpha}(-1,p)$};
\node[vertice] at (P4) {};
\node[right] at (P4) {$\frac{1}{\beta}(-1,q)$};
\end{scope}

\begin{scope}[xshift=16cm,yscale=-1,yshift=-1cm]
\coordinate (R) at (0,2);
\coordinate  (P) at (-2,0) ;
\coordinate (Q) at (2,0);

\draw ($1.1*(P)-.1*(Q)$) -- node[above] {$E_Z$} ($1.1*(Q)-.1*(P)$);
\draw ($1.1*(P)-.1*(R)$) -- node[left] {$E_Y$} ($1.1*(R)-.1*(P)$);
\draw ($1.1*(R)-.1*(Q)$) --  node[right] {$E_X$} ($1.1*(Q)-.1*(R)$);
\node[vertice] at (R) {};
\node[below=3pt] at (R) {$\frac{1}{r}(p,q)$};
\node[vertice] at (Q) {};
\node[above] at (Q) {$\frac{1}{q}(1,\beta)$};
\node[vertice] at (P) {};
\node[above] at (P) {$\frac{1}{p}(1,\alpha)$};
\end{scope}
\end{tikzpicture}
\caption{Birational map from $\bp^2$ to $\bp^2_\omega$}
\label{fig:birat_omega}
\end{center}
\end{figure}
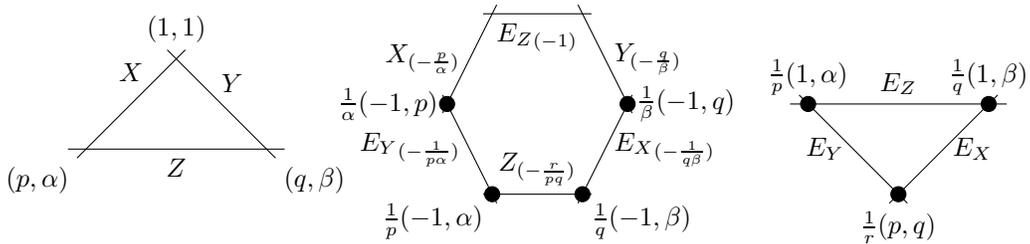
Let us fix $\alpha,\beta\in\ZZ_{>0}$ such that $p q+r=p\beta+q\alpha$. Such integers
exist because of their coprimality and the properties of the semigroup generated by $p,q$.
Let us consider the weighted blow-ups of $\bp^2$ with the weights of the left part of 
Figure~\ref{fig:birat_omega}. Let $S$ be the surface in the middle which, has four
singular points; the self-intersection of the strict transform of~$Z$ is given by
the choice of $\alpha,\beta$. Note that this surface can be obtained from $\bp^2_\omega$
using the weighted blow-ups with weights indicated in the right part of Figure~\ref{fig:birat_omega}.
Note that $\frac{1}{p}(1,\alpha)=\frac{1}{p}(q,r)$ and $\frac{1}{q}(1,\beta)=\frac{1}{q}(p,r)$.
\end{ejm}

\subsection{The projection formula}
This section will mainly follow the results presented in~\cite{Sakai84} for normal surfaces, where proofs can be found.
Recall that $\Div(S)$ is the free abelian group with basis the irreducible
projective curves in~$S$, i.e., the irreducible Weil divisors (in the smooth case
the concepts of Weil and Cartier divisors coincide).
Given a divisor $D\in \Div(S)$ on a projective normal surface $S$ and $S_0:=S\setminus \Sing(S)$,
where $i:S_0\hookrightarrow S$ is the inclusion, the divisor $D$ defines a 
coherent reflexive sheaf of rank one $\cO_S(D)=i_*(\cO_{S_0}(D|_{S_0}))$ which can be given as $\pi_*(\cO_{\bar S}(\bar D))$,
where $\bar D$ is the strict transform of $D$ by a resolution~$\pi$; let $\cE$ be the \emph{exceptional locus}, that is, the collection 
of all the irreducible exceptional divisors of~$\pi$. By definition, for any $\QQ$-divisor
$D=\sum_{i=1}^n a_i D_i\in \Div_\QQ(S)=\QQ\otimes_\ZZ \Div(S)$ we consider $\cO_S(D):=\cO_S(\lfloor D\rfloor)$, 
where $\lfloor D\rfloor=\sum \lfloor a_i\rfloor D_i$. The definition of $\pi^*D$ is given as
$$
\pi^*D=\bar D+\sum_{E\in\cE} m_E E
$$
where $\bar{D}$ is the strict transform of~$D$ and $m_E\in\QQ$ are rational numbers satisfying 
$\bar D\cdot E'+\sum_{E'\in\cE} m_{E} E\cdot E'=0$. Since the intersection matrix $M_\cE=(E\cdot E')_{E,E'\in\cE}$ 
is negative definite, the $m_E$'s are unique. One has the following generalization of the projection formula.

\begin{prop}[{Projection formula~\cite[Thm.~2.1]{Sakai84}}]
\label{prop:projformula}
If $D\in \Div_\QQ(S)$ and $\pi:\bar S\to S$ a resolution of singularities of $S$, then 
$$
\pi_*(\cO_{\bar S}(\pi^* D))\cong \cO_S(D).
$$
\end{prop}

\begin{prop}\label{prop:H0}
Let $S$ be a normal surface and $D \in \Div_\QQ(S)$. Then, the cohomology
group $H^ 0(S, \cO_S(D))$ can be identified with
$\{ h \in \bc(S) \mid (h) + D \geq 0 \}$.
\end{prop}

\begin{proof}
The statement is true for an integral divisor on a smooth variety.
Let $\pi: \bar S \to S$ be a resolution of $S$.
Then, by Proposition~\ref{prop:projformula} (the projection formula) one has 
$\cO_S(D) = \pi_{*} \cO_{\bar S}(\pi^{*} D) = \pi_{*} \cO_{\bar S}(\floor{\pi^{*} D})$
and thus $H^0 (S,\cO_S(D)) = H^0(\bar S,\cO_{\bar S}(\floor{\pi^{*} D}))$. 
Since the latter is an integral divisor in a smooth variety,
$$
H^0(\bar S,\cO_{\bar S}(\floor{\pi^{*} D})) = \{ h \in \bc(\bar S) \mid (h) + \floor{\pi^* D} \geq 0 \}.
$$

The condition $(h) + \floor{\pi^* D} \geq 0$ is equivalent to $(h) + \pi^{*} D \geq 0$ because $(h)$ is an integral divisor.
Writing $h = \pi^{*} h_1$ where $h_1 \in \bc(S)$, one can split the previous condition $(\pi^{*} h_1) + \pi^{*} D \geq 0$
into two different ones,
\begin{equation}\label{eq:global_sections}
(h_1) + D \geq 0 \quad \text{ and} \quad \ord_E \left( (\pi^{*} h_1) + \pi^{*} D \right) \geq 0,
\end{equation}
for all $E\in\cE$.

Note that if $A = (a_{ij})_{i,j}$ is a negative definite real matrix such that $a_{ij} \geq 0$, $i\neq j$,
then $-A^{-1}$ has all non-negative entries. This implies that the pull-back of an effective divisor
is also an effective divisor. Hence one can easily observe that second condition in~\eqref{eq:global_sections}
is implied by the first one and the proof is complete.
\end{proof}

In more generality, given $\varphi:\Sigma \to S$ a birational morphism between normal surfaces and $\cE$ 
its exceptional locus, one can define $\varphi_*(C)$ for any given irreducible divisor $C\in \Div(\Sigma)$ as 0
if $C\in \cE$ or $\varphi(C)$ otherwise. 
This can be extended by linearity to $\Div_\QQ(\Sigma)$. Also, if $D\in \Div_\QQ(S)$, 
then $\varphi^*(D)$
is defined as in the case of a resolution~$\pi$ since for $\varphi$ the intersection matrix
of exceptional irreducible divisors is still negative definite.
One has the following generalization of the projection formula.

\begin{prop}[{\cite[Thm.~6.2]{Sakai84}}]
If $D\in \Div_\QQ(S)$ and $\varphi:\Sigma\to S$ is a birational morphism between normal surfaces and $Z$ an effective
divisor supported on the exceptional locus $\cE$ of $\varphi$, then 
$$
\varphi_*(\cO_\Sigma(\varphi^*D+Z))\cong \cO_S(D).
$$
\end{prop}

\subsection{Picard group for normal surfaces}
\label{sec:picard}

The Picard group $\Pic(S)$ of a smooth projective surface~$S$  
is the quotient of $\Div(S)$ by the subgroup of linear divisors, i.e.
the divisors of the non-zero meromorphic functions on~$S$. The following
exact sequence holds:
\[
0\to\bc^*\to \bc(S)^*\to \Div(S)\to \Pic(S)\to 0.
\]
Two divisors $D_1,D_2\in \Div(S)$ that define the same class in $\Pic(S)$ are 
called \emph{linearly equivalent} and denoted $D_1\sim D_2$. 
Linear equivalence
can be extended to $\Div_\QQ(S)$ as follows: $D_1,D_2\in \Div_\QQ(S)$ are 
linearly equivalent if $D_1-D_2\in \Div(S)$ and $D_1-D_2\sim 0$. 
Analogously, two divisors $D_1,D_2\in \Div(S)$ are called \emph{numerically equivalent}
and denoted $D_1\simmap{\textrm{num}} D_2$ if $D_1\cdot C=D_2\cdot C$ for any $C\in \Div(S)$.
Numerical equivalence can also be extended to $\QQ$-divisors.

Let $\cD$ be a subset of Weil divisors of~$S$ (a normal projective surface)
and let $\Div(S,\cD)$ be the free abelian subgroup of $\Div(S)$ with
basis $\cD$. Let $\bc_\cD(S)^*$ be the multiplicative subgroup of functions~$h$
such that $(h)\in\Div(S,\cD)$. It is clear that the cokernel of
the divisor map $\bc_\cD(S)^*\to\Div(S,\cD)$ is a subgroup of $\Pic(S)$.
In particular, if the images of $\cD$ in $\Pic(S)$ form a generating system,
then the divisors of a generating system of $\bc_\cD(S)^*$ are a complete system
of relations for $\Pic(S)$; conversely, given a generating system $\cD$ and
a complete system of relations for $\Pic(S)$, the functions coming from these relations
are a generating system of $\bc_\cD(S)^*$.

\begin{ejm}
Let $S=\bp^2$; the irreducible divisors are the irreducible plane curves.
Let $H$ be any line; for any curve $C_d$ of degree~$d$, it is
well known that $d H- C_d$ is a linear divisor, i.e., $H$ generates $\Pic(\bp^2)$.
Since $K^*_{H}(\PP^2)=\CC^*$, this implies
$\Pic(\bp^2)=\bz H$ as it is well known. 
With the same ideas $\Pic(\bp^1\times\bp^1)$ is isomorphic to $\bz^2$
generated by the factors.
\end{ejm}

\begin{prop}
\label{prop:pic1}
Let $\varphi:\Sigma\to S$ be a birational morphism; let $\cD$ be
a set of divisors of $S$ generating $\Pic(S)$ and let $\cR$ be a set
of relations such that $\Pic(S)\cong\Div(S,\cD)/\langle\cR\rangle$.
For each $R\in\cR$, let $h_R\in K_\cD^*(S)\subset K^*(S)$ be a function such that $R=(h_R)$.
Let $\tilde{h}_R=\varphi^* h_R\in K^*(\Sigma)$.
Then, if $\cE$ is the set of exceptional locus of~$\varphi$ and $\tilde\cD$ is the set of strict
transforms of the divisors in $\cD$, then
\[
\Pic(\Sigma)\cong\Div(\Sigma,\cE\cup\tilde\cD)/\langle(\tilde{h}_R)\mid R\in\cR\rangle
\]
\end{prop}

\begin{proof}
Let us consider the natural map $\Div(\Sigma,\cE\cup\tilde\cD)\to\Pic(\Sigma)$. Let us consider
$D\in\Div(\Sigma)$. Note that $D=\varphi^*\varphi_*D+E_D$ where $E_D\in\Div(\Sigma,\cE)$. There is a divisor
$D_1\in\Div(S,\cD)$ such that $\varphi_*D\sim D_1$ and as a consequence $D\sim\varphi^*D_1+\tilde{E}$ where
$\tilde{E}\in\Div(\Sigma,\cE)$ and $\varphi^*D_1\in\Div(\Sigma,\tilde{E})$. Hence the above natural map is surjective.

The relations are given by functions in $\bc(\Sigma)^*$ whose divisors are in $\Div(\Sigma,\cE\cup\tilde\cD)$
and those ones are generated by $h_R=\varphi^*h_R$, $R\in\cR$.
\end{proof}

\begin{obs}
Let $R=\sum_{D\in\cD}m_D D$; let us denote $\tilde{D}$ its strict transform.
Then 
\[
(\tilde{h}_R)=\sum_{D\in\cD} m_D \varphi^*(D)=\sum_{D\in\cD} m_D\tilde{D}+\sum_{E\in\cE} m_E E;
\]
the coefficients $m_E$ are integers, while it is possible that the
coefficient of $E$ in each particular $\varphi^*(D)$ is non integer.
\end{obs}

The proposition above has a simple converse.

\begin{prop}
\label{prop:pic2}
Let $\varphi:\Sigma\to S$ be a composition of weighted blow-ups; let $\cE$ be the exceptional locus of~$\varphi$ and 
let $\cD$ be a set of divisors of $\Sigma$ generating $\Pic(S)$ 
and containing $\cE$; let $\cR$ be a set of relations such that $\Pic(\Sigma)\cong\Div(\Sigma,\cD)/\langle\cR\rangle$.
Let $\check{\cD}$ be the set of $\varphi_*$-images of the elements of $\cD\setminus\cE$.
Let us express $R\in\cR$ as $R=\sum_{D\in\cD} m_D(R) D$.
Then
\[
\Pic(S)\cong\Div(S,\check\cD)\left/\left\langle\left.\sum_{D\in\cD\setminus\cE}m_D(R)\pi_*(D)\right| R\in\cR\right\rangle\right.
\]
\end{prop}

\begin{proof}
It is easily seen that $D_1,D_2\in\Div(\Sigma)$ such that $D_1\sim D_2$ then $\varphi_*D_1\sim\varphi_*D_2$.
Hence, it is clear that $\Div(S,\check\cD)$ generates $\Pic(S)$. Note also that
$\bc_\cD(\Sigma)^*$ equals $\bc_{\check{\cD}}(S)^*$, for we need $\cE\subset\cD$.
\end{proof}

\begin{obs}
The condition $\cE\subset\cD$ is essential in the hypotheses of the proposition; without it, the result does not hold.
\end{obs}

\begin{ejm}
With the notations of Example~\ref{ejm:weighted_Cremona},
we can express
\[
\Pic(\bp^2)=\langle X,Y,Z\mid X=Y=Z\rangle=
\langle X,Y,Z\mid X=Y, q X= q Z, p Y= p Z\rangle
\]
From Proposition~\ref{prop:pic1}
\begin{gather*}
\Pic(S)=\langle X,Y,Z,E_X,E_Y,E_Z\mid 
X+E_Z+p E_Y=Y+E_Z+q E_X, \\
q X+ q E_Z+ pq E_Y = q Z +q\alpha E_Y+q\beta E_X,
p Y+ p E_Z+ pq E_X = p Z +p\alpha E_Y+p\beta E_X\rangle=\\
\langle X,Y,Z,E_X,E_Y,E_Z\mid 
X+p E_Y=Y+q E_X,
q X+ q E_Z+ p\beta E_Y = q Z +r E_Y+q\beta E_X,\\
p Y+ p E_Z+ q\alpha E_X = p Z +p\alpha E_Y+r E_X\rangle.
\end{gather*}
With Proposition~\ref{prop:pic2} we obtain
\begin{gather*}
\Pic(\bp^2_\omega)=\langle E_X,E_Y,E_Z\mid 
p E_Y=q E_X,
q E_Z = r E_Y,
p E_Z =r E_X\rangle
\end{gather*}
which is the standard presentation of $\Pic(\bp^2_\omega)$.
\end{ejm}

\subsection{The canonical cycle for \texorpdfstring{$\mathbf{Q}$}{Q}-resolutions and the Riemann-Roch formula}
\label{sec:ZK}
A canonical cycle~$K_S$ on a normal projective surface~$S$ can be obtained as the direct image
of a canonical cycle in a resolution of~$S$; in particular,
$K_S\in\Div(S)$:
Analogously to the smooth case, a canonical cycle on a surface $S$ with quotient singularities obtained after blowing-up 
satisfies the adjunction formula (c.f.~\cite[p.~716]{Chen-OrbifoldAdjunction})
\begin{equation}
\label{eq:adjunction}
-K_S\cdot C=C^2 + \chi^{\orb}(C),
\end{equation}
where $\chi^{\orb}(C)$ is the orbifold Euler characteristic of
a \emph{smooth} irreducible divisor $C$, which in our case is simply, 
$$
\chi^{\orb}(C)=2-2g(C)+\sum_{{\Scale[.75]{\array{c} P\in C\\ S_P=\frac{1}{d_P}(1,\cdot)\endarray}}} \left( \frac{1}{d_P} - 1\right),
$$
the point $P$ runs on all singular points of the surface on $C$ and $d_P$ denotes its order as a quotient singularity. Sometimes we need only the numerical
properties of a canonical divisor. This is way 
a \emph{canonical cycle} can be defined as any cycle $Z_K\in \Div_\QQ(S)$ satisfying the numerical conditions~\eqref{eq:adjunction}, 
that is, $Z_K\cdot C=C^2 + \chi^{\orb}(C)$ for all \emph{smooth} 
irreducible divisors~$C$. Note that, even if a canonical cycle had integer 
coefficients, it may not be an anti-canonical divisor.
The anti-canonical divisor is numerically equivalent to the canonical cycle $-K_S\ \simmap{\textrm{num}}\ Z_K$. 

A useful tool to calculate $h^i(\cO_S(D)) := h^i(S,\cO_S(D))$ is the Riemann-Roch formula for normal surfaces 
\begin{equation}
\label{eq:RR}
\chi(S,\cO_S(D))-\chi(S,\cO_S)=\frac{1}{2}D(D+Z_K)-\Delta_S(-D),
\end{equation}
where $\Delta_S(-D)$ is a correction term that deserves some special attention. This tool is combined with Serre Duality~\cite[Theorem~3.1]{Blache-RiemannRoch}, which implies that 
$h^0(S,\cO_S(D))=h^2(S,\cO_S(K_S-D))$.

\begin{ejm}
Consider $S=\PP^2_{w}$, the weighted projective plane of weight $w=(w_0,w_1,w_2)$. The canonical divisor $K_S=-X_0-X_1-X_2$ has degree 
$-|w|$ for $|w|:=w_0+w_1+w_2$ --\,see~\cite{Dolgachev-weighted}. If $-|w|< k <0$, then $\chi_S(k):=\chi(S,\cO_S(k))=0$ and \eqref{eq:adjunction} becomes
\begin{equation}
\label{eq:adjp2}
\sum_{\{i_0,i_1,i_2\}=\{0,1,2\}}\Delta_{\frac{1}{w_{i_0}}(w_{i_1},w_{i_2})}(k+|w|)=1+\frac{k(k+|w|)}{2\bar w}=:g_{w,k},
\end{equation}
where $\bar w=w_0w_1w_2$.
\end{ejm}

\subsection{The correction term \texorpdfstring{$\Delta_S$}{Delta\_S}}

The correction term $\Delta_S(D)$ of the Riemann-Roch formula for normal surfaces~\eqref{eq:RR} has been considered in different contexts
(cf.~\cite{Laufer-normal,Brenton-RiemannRoch,Blache-RiemannRoch}). Here we will consider the quotient surface singularity case,
where $\Delta_S(D)=\sum_{P\in \Sing S}\Delta_P(D)$ is the sum of locally defined invariants at the singular 
points of $S$ of local type $\CC^2/G$ for a finite group $G$. More explicitly, $\Delta_P(D)$ is a rational map 
$$
\array{rrcl}
\Delta_P(D): &\Weil_P(S)/\Cart_P(S) & \longrightarrow & \QQ
\endarray
$$
defined on the local Weil divisors (formal finite combinations of local irreducible curves) vanishing on the local Cartier 
divisors (associated with principal ideals). This map is explicitly described in \cite{ji-J-curvettes} for cyclic surface singularities.
The quotient $\Weil_P(S)/\Cart_P(S)$ is isomorphic to the (multiplicative) group $\Hom(G,\CC^*)$ of characters on $G$, which in the 
cyclic case is the group of $d$-roots of unity, for $d:=|G|$. Once a choice of this root of unity $\zeta_d$ is given, then we can use 
the notation $\Delta_P(k)$ to describe $\Delta_P(\zeta^k)$ and the standard notation $S_P=\frac{1}{d}(1,p)$ for a local cyclic surface 
singularity $\CC^2/\mu_d$ given by the action $\mu_d(x,y)=(\zeta_d x,\zeta_d^py)$. 

\begin{prop}
\label{prop:Delta}
Let $k\in \ZZ$ and $p,q\in \ZZ_{\geq 0}$ with $d=p+q>0$ and $r,s\in \{0,1\}$. Then
\begin{equation}
\label{eq:delta-dual}
\begin{aligned}
&\Delta_{\frac{1}{d}(1,p)}(k-(r-1)p)+\Delta_{\frac{1}{d}(1,q)}(k-(s-1)q)=
\frac{(1-d)(r+s-2)}{2d}+\left\{\frac{-k}{d}\right\}(r+s-1)
\end{aligned}
\end{equation}
In particular,
\begin{enumerate}[label=\rm(\arabic*)]
\item $\Delta_{\frac{1}{d}(1,p)}(k)+\Delta_{\frac{1}{d}(1,q)}(k+q)=\frac{d-1}{2d}$,
\item $\Delta_{\frac{1}{d}(1,p)}(k)+\Delta_{\frac{1}{d}(1,q)}(k)= \{ \frac{-k}{d} \}$.
\end{enumerate}
\end{prop}

\begin{proof}
Note that $k$ can be considered modulo $d$ and hence it is enough to show the result for $k=1,\dots,d$.
We will use formula~\eqref{eq:adjp2} repeatedly for the projective planes $S_i=\PP^2_{w_i}$ for 
$w_1=(d,1,p)$, $w_2=(d,1,q)$, and $w_3=(1,p,q)$. Note that
$$
\begin{aligned}
0=\chi_{S_1}(k-1-d-rp) & = g_{w_1,k-1-d-rp}-\Delta_{\frac{1}{d}(1,p)}(k-(r-1)p)-\Delta_{\frac{1}{p}(1,q)}(k)\\
0=\chi_{S_2}(k-1-d-sq) & = g_{w_2,k-1-d-sq}-\Delta_{\frac{1}{d}(1,q)}(k-(s-1)q)-\Delta_{\frac{1}{q}(1,p)}(k)\\
0=\chi_{S_3}(k-1-d) & = g_{w_3,k-1-d}-\Delta_{\frac{1}{p}(1,q)}(k)-\Delta_{\frac{1}{q}(1,p)}(k).
\end{aligned}
$$
Subtracting the third equation from the sum of the first two, one obtains 
$$g_{w_1,k-1-d-rp}+g_{w_2,k-1-d-sq}-g_{w_3,k-1-d}=\Delta_{\frac{1}{d}(1,p)}(k-(r-1)p)+\Delta_{\frac{1}{d}(1,q)}(k-(s-1)q).$$
Also note that
$$
\begin{aligned}
g_{w_1,k-1-d-rp}+g_{w_2,k-1-d-sq}-g_{w_3,k-1-d} & =
\frac{pr^2+qs^2+(q-2k+1)r+(p-2k+1)s+2k-2}{2d}\\
&=\frac{(d-2k+1)(r+s)+2k-2}{2d}
\end{aligned}
$$
since $r^2=r$ and $s^2=s$.

Note that this equals the right-hand side of~\eqref{eq:delta-dual} in case $k\in \{1,\dots,d\}$.
The result follows for general $k\in \ZZ$ using that $-\frac{\bar k}{d}=\left\{ -\frac{k}{d} \right\}-1$, where 
$\bar k$ denotes the remainder class of $k$ in $\{1,\dots,d\}$.
\end{proof}

\begin{ejm}
In the particular case $k=0$, notice that $\Delta_{\frac{1}{d}(1,p)}(p)=\frac{d-1}{2d}$, since $\Delta_{\frac{1}{d}(1,q)}(0)=0$.
Multiplying by $\bar p$ (which amounts to choosing a different primitive $d$-root of unity) one obtains 
$\Delta_{\frac{1}{d}(\bar p,1)}(1)=\Delta_{\frac{1}{d}(1,\bar p)}(1)=\frac{d-1}{2d}$. Since this is true for 
any $p$ prime with $d$, one obtains $\Delta_{\frac{1}{d}(1,p)}(1)=\frac{d-1}{2d}$.

In particular, if $d=2$, then $\Delta_{\frac{1}{2}(1,1)}(0)=0$ and $\Delta_{\frac{1}{2}(1,1)}(1)=\frac{1}{4}$.
\end{ejm}

\section{Rational ruled toric surfaces}
\label{sec:RRTsurfaces}

Following~\cite{Sakai88} a \emph{rational ruled fibration} over a normal projective surface $S$ is a surjective 
morphism $\pi:S\to \PP^1$ whose generic fiber is isomorphic to $\PP^1$. Moreover, the fibration is called \emph{minimal}
if its fibers are irreducible. In this spirit one can define

\begin{dfn}
\label{def:trrs}
A \emph{rational ruled toric surface} is a projective normal surface $S$ with quotient singularities with and a minimal 
rational ruled fibration $\pi:S\to\bp^1$ with two marked fibers $F_1, F_2$ and two marked disjoint sections $S_1, S_2$. 
The toric structure of $S$ comes from the two-dimensional torus (projecting onto $\PP^1\setminus \{\pi(F_1),\pi(F_2)\}=\bc^*$ 
with $\bc^*$-fiber given by $\pi^{-1}(P)\setminus \{S_1\cap \pi^{-1}(P),S_2\cap \pi^{-1}(P)\}$), four $1$-dimensional 
tori (the two fibers and the two sections outside their intersection) and the $4$ vertices $F_i\cap S_j$ 
(containing the singularities of~$S$).
\end{dfn}

The following statement is an immediate consequence of Theorem 1.2 in~\cite{Sakai88}.

\begin{lema}\label{lema:sym}
Let $C\subset S$ be a 
a singular fiber of a rational ruled toric surface.
Then $\#(C\cap\sing S)=2$
and the two points are of type $\frac{1}{d}(1,m)$ and $\frac{1}{d}(1,-m)$.
\end{lema}

If the surface $S$ is smooth, we are considering a Hirzebruch surface $\Sigma_n$, two fibers, the negative
self-intersection section (if $n>0$) and a section with self-intersection~$n$ (two null self-intersections
if $n=0$).

\begin{ejm}
Let us fix three lines $A,B,C$ in general position in $\bp^2$, and let $p:=A\cap B$.
Let $\pi:\Sigma_1\to\bp^2$ be the blowing-up of~$p$. Then, the projection from~$p$
induces a rational ruled toric surface with the fibers $A,B$ and the sections $C,E$
(where~$E$ is the exceptional divisor of~$\pi$).
\end{ejm}

\begin{ejm}
The above example can be easily generalized. Keep the notation
of Example~\ref{ejm:weighted_Cremona}.

If $p_z:=X\cap Y$, then $(\bp^2_\omega,p_z)$ is a singular point
of type $\frac{1}{r}(p_X,q_Y)$. The $(p,q)$-weighted blow-up of $p_z$ is a map
$\rho:\Sigma\to\bp^2_\omega$ which is a rational ruled toric surface for
the fibers $X,Y$ and the sections $Z,E$; there are two singular points
on~$E$ of type $\frac{1}{q}(p_X,-r_E)$ and $\frac{1}{p}(q_Y,-r_E)$. The following holds:
\[
E^2=-\frac{r}{pq}=-Z^2\quad 
X^2=Y^2=0.
\]
Note the symmetries in the singulars point lying on a fiber.
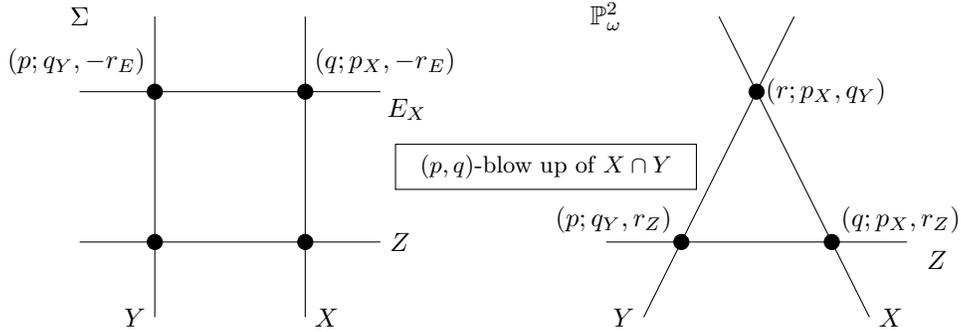
\begin{figure}[ht]
\begin{center}
\begin{tikzpicture}[vertice/.style={draw,circle,fill,minimum size=0.2cm,inner sep=0}]
\tikzset{%
  suma/.style args={#1 and #2}{to path={%
 ($(\tikztostart)!-#1!(\tikztotarget)$)--($(\tikztotarget)!-#2!(\tikztostart)$)%
  \tikztonodes}}
} 
\coordinate (X) at (0,2);
\coordinate (Z) at (-1,0);
\coordinate (Y) at (1,0);
\coordinate (X1) at (-1,1);
\draw[suma=.5 and .5] (Y) to node[pos=-.1,below] {$Z$} (Z);
\draw[suma=.5 and .5] (X) to (Y) node[right] {$X$};
\draw[suma=.5 and .5] (X) to (Z) node[left] {$Y$};

\node[vertice] at (X) {};
\node[right] at (X) {$(r;p_X,q_Y)$};
\node[vertice] at (Y) {};
\node[above right] at (Y) {$(q;p_X,r_Z)$};
\node[vertice] at (Z) {};
\node[above left] at (Z) {$(p;q_Y,r_Z)$};

\node[] at ($(X)+(-2,1)$) {$\mathbb{P}^2_\omega$};

\begin{scope}[xshift=-7cm]
\coordinate (X) at (0,2);
\coordinate (Z) at (-1,0);
\coordinate (Y) at (1,0);
\coordinate (X1) at (-1,1);
\coordinate (XZ) at ($(X)+(Z)$);
\coordinate (XY) at ($(X)+(Y)$);

\draw[suma=.5 and .5] (Y) to node[right=2cm] {$Z$} (Z);
\draw[suma=.5 and .5] (XY) to (Y) node[right] {$X$};
\draw[suma=.5 and .5] (XZ) to (Z) node[left] {$Y$};
\draw[suma=.5 and .5] (XZ) to (XY) node[below right=-1pt] {$E_X$};

\node[vertice] at (XY) {};
\node[vertice] at (XZ) {};
\node[vertice] at (Y) {};
\node[vertice] at (Z) {};

\node[] at ($(X)+(-2,1)$) {$\Sigma$};
\draw (2.2,.7) rectangle +(4,.6);
\node[left=-10pt] at ($(X)+(5.65,-1)$) {\small $(p,q)$-blow up of $X\cap Y$};

\node[right=1cm] at ($1.2*(X)$) {$(q;p_X,{-r}_E)$};
\node[left=1cm] at  ($1.2*(X)$){$(p;q_Y,-r_E)$};

\end{scope}

\end{tikzpicture}

 \caption{Weighted blow-up in $\bp^2_\omega$}
\label{fig:weigh-bu}
\end{center}
\end{figure}
\end{ejm}

\begin{prop}\label{prop:condiciones}
Let $S$ be a rational ruled toric surface. Let $S_+,S_-$ be the sections
and $F_1,F_2$ be the fibers.
\begin{enumerate}
\enet{\rm(\arabic{enumi})} 
\item The types of the singular points $F_i\cap S_\pm$ are of the form
$\frac{1}{d_i}(1,\pm{n_i})$ corresponding to the equations of ${S_\pm}$ and $F_i$ respectively.
\item $S_+^2=-S_-^2=:r$.
\item $\frac{n_1}{d_1}+\frac{n_2}{d_2}-r\in\bz$.
\end{enumerate}

\end{prop}

\begin{proof}
The first statement is a direct consequence of Lemma~\ref{lema:sym}.
\end{proof}

\begin{thm}\label{thm:class_toric}
For any $r\in\bq_{\geq 0}$, $0\leq n_i<d_i$ ($d_i\in\bn$, $n_i\in\bz_{\geq 0}$, $\gcd(n_i,d_i)=1$) 
there is a unique rational ruled toric surface with these invariants.
\end{thm}

For the proof, we are going to introduce the concept of weighted Nagata transformations.

\subsection{Weighted Nagata transformations}
\label{sec:Nagata}

Let us consider $\Sigma$ a smooth ruled surface (with base~$\bp^1$ for simplicity)
and let $P\in\Sigma$. The Nagata transformation of $\Sigma$ based on~$P$ is 
a birational map constructed as follows. First, we blow up the point~$P$ to
obtain a surface~$\hat\Sigma$; let
$F\subset\Sigma$ the fiber containing~$P$; note that $(F\cdot F)_{\Sigma}=0$
and hence $(F\cdot F)_{\hat\Sigma}=-1$. By Castelnuovo criterion, we can contract 
$F$ to obtain a new surface~$\tilde{\Sigma}$ which is also ruled. Let us
assume that $\Sigma\cong\Sigma_n$, $n\geq 0$; if $P$ belongs to a section
with non-positive self-intersection then $\tilde\Sigma\cong\Sigma_{n+1}$; such a section is unique if $n>0$ and 
the condition is always satisfied if $n=0$. If the condition is not satisfied,
then $\tilde{\Sigma}\cong\Sigma_{n-1}$. We are going to generalize this notion;
the generalization will work for singular ruled surfaces and it will allow
for smooth ones changes $n\rightarrow n\pm k$, with $k>1$.

\subsection{Construction of the surfaces \texorpdfstring{$S_{(d_1,d_2,p_1,q_2,r)}$}{S(d1,d2,p1,p2,r)}}
\label{sec:construction}

One can produce these surfaces starting from the Hirzebruch surface $\Sigma_1$ and after performing
two weighted Nagata transformations as follows (see Figure~\ref{fig:wnagata}). 

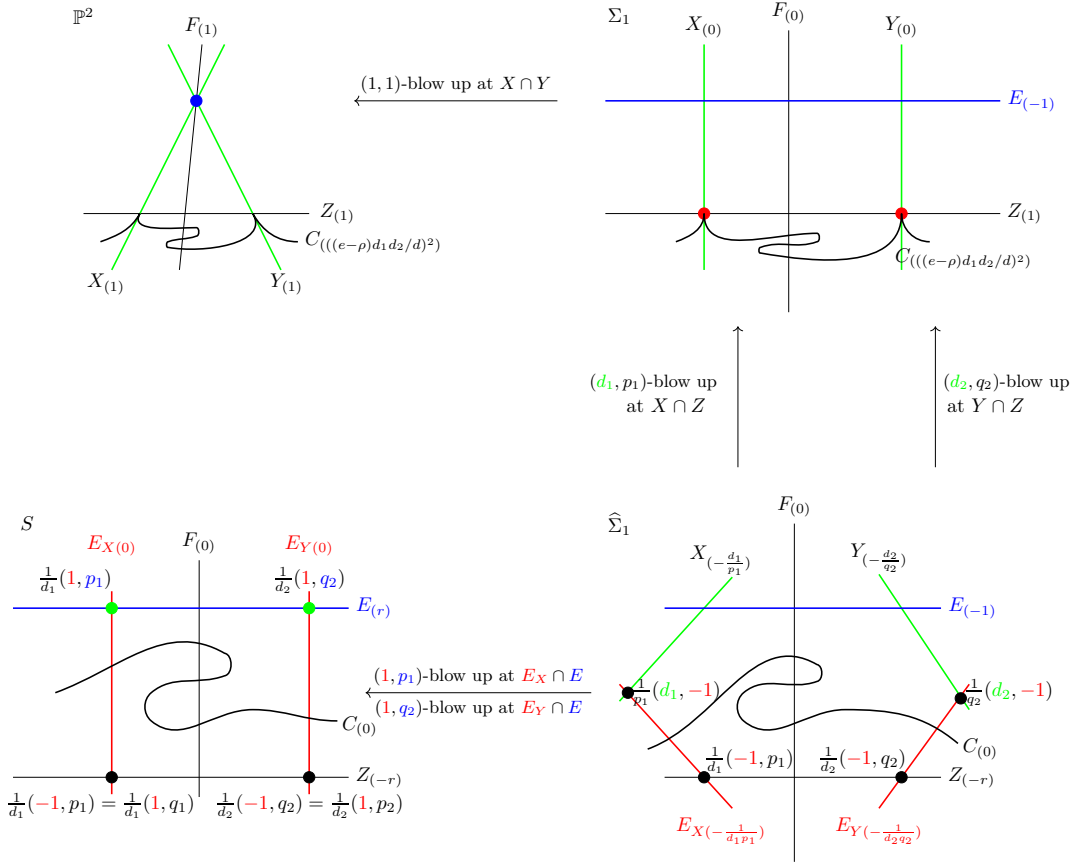
\begin{figure}[ht]
\begin{center}
\scalebox{.75}{
\begin{tikzpicture}[vertice/.style={draw,circle,fill,minimum size=0.2cm,inner sep=0}]
\tikzset{%
  suma/.style args={#1 and #2}{to path={%
 ($(\tikztostart)!-#1!(\tikztotarget)$)--($(\tikztotarget)!-#2!(\tikztostart)$)%
  \tikztonodes}}
}
\begin{scope}[xshift=-1cm]
\coordinate (p4) at (-.2,0);
\coordinate (p3) at (0,2);
\coordinate (p2) at (-1,0);
\coordinate (p1) at (1,0);
\coordinate (X1) at (-1,1);
\node[] at ($(p3)+(-2,1.5)$) {$\mathbb{P}^2$};
\draw[suma=.5 and .5] (p1) to node[shift={(2.5,0)}] {$Z_{(1)}$} (p2);
\draw[suma=.5 and .5,green,thick] (p3) to (p1) node[shift={(.1,-.25)},black] {$Y_{(1)}$};
\draw[suma=.5 and .5,green,thick] (p3) to (p2) node[shift={(-.1,-.25)},black] {$X_{(1)}$};
\draw[suma=.5 and .5] (p3) to (p4) node[shift={(.4,4.3)}] {$F_{(1)}$};
\draw[black,smooth,thick] (-1.7,-.5) to [out=20,in=240] (p2) to [out=240,in=180] (0,-.25) to [out=0,in=0] (0,-.4) to [out=200,in=20] (-.5,-.5) to [out=180,in=200] (-.5,-.6) to [out=-10,in=300] (p1) to [out=300,in=180] (1.8,-.5) node [right] {$C_{(((e-\rho)d_1d_2/d)^2)}$};
\node[vertice,blue] at (p3) {};
\end{scope}

\draw[arrows=<-] (1.8,2) -- node[above] {\small $(1,1)$-blow up at $X\cap Y$} (5.4,2) ;

\begin{scope}[xshift=9cm]
\coordinate (X) at (0,2);
\coordinate (Z) at (-1,0);
\coordinate (Y) at (2.5,0);
\coordinate (X1) at (-1,1);
\coordinate (XZ) at ($(X)+(Z)$);
\coordinate (XY) at ($(X)+(Y)$);
\draw[suma=.5 and .5] (Z) to (Y) node[right] {$Z_{(1)}$};
\draw[suma=.5 and .5] (.5,2) to (.5,-.5) node[above=5cm] {$F_{(0)}$};
\draw[suma=.5 and .5,green,thick] (XY) to (Y) node[above=4cm,black] {$Y_{(0)}$};
\draw[suma=.5 and .5,green,thick] (XZ) to (Z) node[above=4cm,black] {$X_{(0)}$};
\draw[suma=.5 and .5,blue,thick] (XZ) to (XY) node[right] {$E_{(-1)}$};
\node[vertice,red] at (Y) {};
\node[vertice,red] at (Z) {};
\draw[black,smooth,thick] (Z)+(-.5,-.5) to [out=20,in=270] (Z) to [out=270,in=180] (1,-.35) to [out=0,in=0] (1,-.5) to [out=200,in=20] (0,-.6) to [out=180,in=200] (0,-.7) to [out=-10,in=270] (Y) to [out=270,in=180] (3,-.5) node [shift={(.6,-.3)}] {$C_{(((e-\rho)d_1d_2/d)^2)}$};

\node[] at ($(X)+(-2.5,1.5)$) {$\Sigma_1$};

\end{scope}

\node[left=-10pt] at ($(8,-3)$) {\small $({\color{green}{d_1}},{\color{black}{p_1}})$-blow up};
\node[left=-10pt] at ($(7.8,-3.4)$) {at $X\cap Z$};
\draw[arrows=<-] (8.6,-2) -- (8.6,-4.5);

\node[left=-0pt] at ($(14.6,-3)$) {\small $({\color{green}{d_2}},{\color{black}{q_2}})$-blow up};
\node[left=-0pt] at ($(13.8,-3.4)$) {at $Y\cap Z$};
\draw[arrows=<-] (12.1,-2) -- (12.1,-4.5);

\begin{scope}[xshift=9cm,yshift=-9cm]
\coordinate (X) at (0,2);
\coordinate (Z) at (-1,-1);
\coordinate (Y) at (2.5,-1);
\coordinate (Y1) at (3.55,.4);
\coordinate (X1) at (-2.35,.5);
\coordinate (XZ) at (-1,2);
\coordinate (XY) at (2.5,2);
\draw[suma=.2 and .2] (Z) to (Y) node[right] {$Z_{(-r)}$};
\draw[suma=.5 and .5] (.6,2) to (.6,-1) node[above=6cm] {$F_{(0)}$};
\draw[suma=.5 and .5,green,thick] (XZ) to (-2,.9) node[shift={(1.8,2.5)},black] {$X_{(-\frac{d_1}{p_1})}$};
\draw[suma=.5 and .5,red,thick] (Z) to (-2,.1) node[shift={(1.8,-2.6)}] {$E_{X(-\frac{1}{d_1p_1})}$};
\draw[suma=.5 and .5,red,thick] (Y) to (3.3,.1) node[shift={(-1.6,-2.6)}] {$E_{Y(-\frac{1}{d_2q_2})}$};
\draw[suma=.5 and .5,green,thick] (XY) to (3.3,.8) node[shift={(-1.6,2.7)},black] {$Y_{(-\frac{d_2}{q_2})}$};
\draw[suma=.2 and .2,blue,thick] (XZ) to (XY) node[right] {$E_{(-1)}$};
\node[vertice] at (Y) {} node[shift={(-.2,-.7)}] {$\frac{1}{d_1}({\color{red}{-1}},{\color{black}{p_1}})$};
\node[vertice] at (Z) {} node[shift={(1.8,-.7)}] {$\frac{1}{d_2}({\color{red}{-1}},{\color{black}{q_2}})$};
\node[vertice] at (X1) {} node[shift={(4.4,.5)}] {$\frac{1}{q_2}({\color{green}{d_2}},{\color{red}{-1}})$};
\node[vertice] at (Y1) {} node[shift={(-1.5,.5)}] {$\frac{1}{p_1}({\color{green}{d_1}},{\color{red}{-1}})$};
\draw[black,smooth,thick] (-2,-.5) to [out=20,in=150] (1,1) to [out=-10,in=10] (1,.7) to [out=200,in=40] (-.3,.5) to [out=220,in=140] (-.3,0) to [out=-40,in=180] (1.5,.2)  to [out=0,in=140] (3.5,-.4) node [shift={(.4,-.1)}] {$C_{(0)}$};

\node[] at ($(X)+(-2.5,1.5)$) {$\widehat\Sigma_1$};

\end{scope}

\draw[arrows=<-] (2,-8.5) --node[above] {\small $({\color{red}{1}},{\color{blue}{p_1}})$-blow up at ${\color{red}{E_X}}\cap {\color{blue}{E}}$}
node[below] {\small $({\color{red}{1}},{\color{blue}{q_2}})$-blow up at ${\color{red}{E_Y}}\cap {\color{blue}{E}}$} (6,-8.5);

\begin{scope}[xshift=-1.5cm,yshift=-9cm]
\coordinate (X) at (0,2);
\coordinate (Z) at (-1,-1);
\coordinate (Y) at (2.5,-1);
\coordinate (X1) at (-1,1);
\coordinate (XZ) at (-1,2);
\coordinate (XY) at (2.5,2);
\draw[suma=.5 and .2] (Z) to (Y) node[right] {$Z_{(-r)}$};
\draw[suma=.1 and .1] (.55,-1) to (.55,2.5) node[above] {$F_{(0)}$};
\draw[suma=.1 and .1,red,thick] (Y) to (XY) node[above=.5cm] {$E_{Y(0)}$};
\draw[suma=.1 and .1,red,thick] (Z) to (XZ) node[above=.5cm] {$E_{X(0)}$};
\draw[suma=.5 and .2,blue,thick] (XZ) to (XY) node[right] {$E_{(r)}$};
\node[vertice,green] at (XY) {} node[shift={(-1.65,2.5)}] {$\frac{1}{d_1}({\color{red}{1}},{\color{blue}{p_1}})$};
\node[vertice,green] at (XZ) {} node[shift={(2.5,2.5)}] {$\frac{1}{d_2}({\color{red}{1}},{\color{blue}{q_2}})$};
\node[vertice] at (Y) {} node[shift={(-1.2,-1.5)}] {$\frac{1}{d_1}({\color{red}{-1}},{\color{black}{p_1}})=\frac{1}{d_1}({\color{red}{1}},{\color{black}{q_1}})$};
\node[vertice] at (Z) {} node[shift={(2.5,-1.5)}] {$\frac{1}{d_2}({\color{red}{-1}},{\color{black}{q_2}})=\frac{1}{d_2}({\color{red}{1}},{\color{black}{p_2}})$};
\draw[black,smooth,thick] (-2,.5) to [out=20,in=150] (1,1.2) to [out=-10,in=10] (1,.8) to [out=200,in=40] (-.3,.5) to [out=220,in=140] (-.3,0) to [out=-40,in=180] (1.5,.2)  to [out=0,in=180] (3,0) node [shift={(.4,-.1)}] {$C_{(0)}$};
\node[] at ($(X)+(-2.5,1.5)$) {$S$};
\end{scope}
\end{tikzpicture}
}
 \caption{Construction of $S_{r,(d_1,p_1),(d_2,q_2)}$}
\label{fig:wnagata}
\end{center}
\end{figure}

As a brief explanation of Figure~\ref{fig:wnagata}, we have considered two positive integers $p_1,q_2\in \ZZ_{>0}$ such that 
\begin{equation}
\label{eq:emes}
\frac{p_1}{d_1}+\frac{q_2}{d_2}-r=1.
\end{equation}
Hence,
\begin{equation}
\label{eq:pq0} 
d_2 p_1+d_1 q_2=d_1 d_2+rd_1d_2\geq d_1d_2.
\end{equation}

In order to find $p_1$ and $q_2$ let us start with the equality $\frac{n_1}{d_1}+\frac{n_2}{d_2}-r=1-k\in \ZZ$ 
from Proposition~\ref{prop:condiciones} where $0<n_i<d_i$ and $r\geq 0$. Under these conditions $k\geq 0$. Then, for instance, $p_1:=n_1$ and 
$q_2:=n_2+kd_2>0$ satisfy~\eqref{eq:emes}.

The purpose is to construct the rational ruled toric surface $S$ associated with $d_i,n_i,r$ in Proposition~\ref{prop:condiciones}.
Self-intersection of divisors is shown in parenthesis. 
Also, in order to simplify notation, each blow-up and each singular point are color coded in order to specify the direction.
For instance, $\frac{1}{d}({\color{red}{a}},{\color{blue}{b}})$ represents a cyclic singular point of order $d$ whose action 
in local coordinates is given by $(x,y)\mapsto (\zeta^a x,\zeta^b)$ where $\{x=0\}$ (resp. $\{y=0\}$) is the local equation the 
divisor colored in red (resp. blue) --\,see Notation~\ref{ntc:cyclic}. Analogously, a weighted blow-up of type 
$({\color{red}{a}},{\color{blue}{b}})$ means \emph{relative} to $A$, $B$ where $A$ (resp. $B$) is the divisor colored in red (resp. blue).
The resulting self-intersections after blow-ups and the types of the singular points on exceptional divisors are a consequence of the 
discussion in section~\ref{sec:blowup-pq} and the intersections formulas summarized in~\eqref{eq:intersections_pq}.

\begin{obs}
\label{rem:special}
If $r=\frac{\rho}{e}$ and $\frac{d_1}{d_2}=\frac{d'_1}{d'_2}$ irreducible fractions, where $d'_i=\frac{d_i}{d}$ and 
$d:=\gcd(d_1,d_2)$, then $d'_1d'_2|e$. 
If $r\in \ZZ$, then $d=d_1=d_2$ and $n_1+n_2=d$ and $S$ can be obtained
as the quotient by $\mu_{d}$ of a smooth ruled surface which is $\bp^1\times\bp^1$ if $r=0$; the action
of $\mu_d$ in an affine chart isomorphic to $\CC^2$ is given 
by $\zeta\cdot(x,y):=(\zeta x,\zeta^{n_1} y)$. When $r=0$ the surface will be called \emph{biruled} (since two rulings exist); otherwise the surface is called \emph{uniruled}.
\end{obs}

Even though it is not necessary for the construction of $S$, note that the strict transform of the rational curve 
$C=\{x^{ep_1d'_2}y^{eq_2d'_1}-z^{(e-\rho)d_1d_2/d}\}\subset\PP^2$ --\,see~\eqref{eq:pq0} and Remark~\ref{rem:special}\,--
is a multi-section of the ruling with zero self-intersection and intersecting the smooth fiber $F$ at $(e-\rho)d_1d_2/d$ points.

\begin{obs}
Conceptually, the surface $S_{(d,p)}$ can also be obtained as the quotient of $\PP^1\times\PP^1$ by the cyclic action of the 
multiplicative group $\mu_d$ as $\zeta\cdot ([x:y],[s:t])=([x:\zeta^p y],[s:\zeta^q t])$. However, the construction presented 
here will be more convenient for calculation purposes.
\end{obs}

\begin{proof}[Proof of Theorem{\rm~\ref{thm:class_toric}}]
Let us start with $\bp^1\times\bp^1$. A sequence of weighted Nagata transformations yield the desired
ruled toric surface. Moreover, from any ruled toric surface, the \emph{inverse} sequence 
of Nagata transformations produces $\bp^1\times\bp^1$.
\end{proof}

\subsection{The Picard group of \texorpdfstring{$S$}{S}}
\label{sec:picardgroup}

Recall that the Picard group $\Pic(S)$ of a surface $S$ is defined as the divisor class group factored out by linear equivalence, 
that is, $D_1\sim D_2$ if $D_1-D_2=\supp (f)$ for a rational global function $f\in \bc(S)$. Given a list of invariants $I=(d_1,d_2,p_1,q_2,r)$ 
we will follow section~\ref{sec:construction} to construct the rational ruled toric $S_I$ and use Propositions~\ref{prop:pic1} and \ref{prop:pic2} 
to describe a presentation of the Picard groups of $S_I$ and all the intermediate surfaces obtained after each blow-up.
\begin{equation}
\label{eq:Picard_S}
\array{rcl}
\Pic(\PP^2)&=&\langle X,Y,Z,F: X=Y=Z=F\rangle\cong \ZZ Z.\\
\Pic(\Sigma_1)&=&\langle X,Y,Z,F,E: X+E=Y+E=Z=F+E\rangle\cong \ZZ Z\times\ZZ E.\\
\Pic(\widehat\Sigma_1)&=&\left\langle X,E_X,Y,E_Y,Z,F,E: 
\begin{cases} X+d_1E_X=Y+d_2E_Y=F, \\Z+p_1E_X+q_2E_Y=F+E\end{cases}\right\rangle\\
&\cong &\ZZ Z\times\ZZ E\times\ZZ E_X\times\ZZ E_Y.\\
\Pic(S)&=&\left\langle E_X,E_Y,Z,F,E: \begin{cases} d_1E_X=d_2E_Y=F, \\Z+p_1E_X+q_2E_Y=F+E\end{cases}\right\rangle\\
&\cong& \langle Z,F,E_X,E_Y\mid F=d_1E_X=d_2E_Y\rangle \cong \ZZ\ Z \times \ZZ\ E_X \times \frac{\ZZ}{d\ZZ}\ T,
\endarray
\end{equation}
where $T=\frac{d_1}{d}E_X-\frac{d_2}{d}E_Y$ and $d:=\gcd(d_1,d_2)$. 
The intersection matrix with respect to the generating system $\{ Z,F,E_X,E_Y\}$ is given as
$$M_S=
\left[
\begin{matrix}
r & 1 & \frac{1}{d_1} & \frac{1}{d_2}\\
1 & 0 & 0 & 0 \\
\frac{1}{d_1} & 0 & 0 & 0\\
\frac{1}{d_2} & 0 & 0 & 0
\end{matrix}
\right]
$$
and the kernel of this bilinear form is generated by $F-d_1E_X$, $F-d_2E_Y$, and the torsion divisor $T$ which is also numerically equivalent to zero. 
In particular $\frac{d_1}{d}E_X$ and $\frac{d_2}{d}E_Y$ are numerically equivalent, however they are not linearly equivalent unless~$d=1$.

Moreover, let $\Gamma$ be the free subgroup generated by $Z,F$; then, the quotient of the Picard group by $\Gamma$ is isomorphic
to the direct product $\bz/d_1\times\bz/d_2$ generated by the classes of $E_X,E_Y$. Hence any element $D$ of the Picard group of~$S$ can be uniquely 
represented as
\[
D\sim a Z+b F+\alpha E_X+\beta E_Y,\quad a,b,\alpha,\beta\in\bz,\quad 0\leq\alpha<d_1,\quad 0\leq\beta<d_2.
\]

\subsection{The canonical cycle of \texorpdfstring{$S$}{S}}

Following the discussion in section~\ref{sec:ZK} the canonical cycle is numerically equivalent to 
\begin{equation}
\label{eq_Knum}
Z_K\ \simmap{\text{{\tiny{num}}}} \ 2Z+rF+E_X+E_Y.
\end{equation}

Using Riemann-Roch one can calculate $\chi(\cO_S(D)) := \chi(S,\cO_S(D))$ as follows

\begin{lema}
\label{lemma:chigen}
Let $D\sim a Z+bF+\alpha E_X+\beta E_Y\in \Pic(S)$, where $0\leq \alpha<d_1$ and $0\leq \beta<d_2$, then
$$
\array{rcl}
\chi(\cO_S(D))&=&
1+\frac{1}{2}a(b + r) + \left(b + \frac{\alpha}{d_1} + \frac{\beta}{d_2} - ar\right)\frac{(a + 2)}{2} + \frac{a(\alpha + 1)}{2d_1} + \frac{a(\beta + 1)}{2d_2}\\
&&-\Delta_{\frac{1}{d_1}(1,q_1)}(-\alpha-aq_1)-\Delta_{\frac{1}{d_2}(1,p_2)}(-\beta-ap_2)\\
&&-\Delta_{\frac{1}{d_1}(1,p_1)}(-\alpha)-\Delta_{\frac{1}{d_2}(1,q_2)}(-\beta).
\endarray
$$
\end{lema}

\begin{proof}
By the discussion in the preceding section
$$
D(D+Z_K)\!=\!(a,b,\alpha,\beta) M_S
\left(
\begin{matrix}
a+2\\
b+r\\
\alpha+1\\
\beta+1
\end{matrix}
\right)
\!=\!
a(b + r) + \left(b + \frac{\alpha}{d_1} + \frac{\beta}{d_2} - ar\right)(a + 2) + \frac{a(\alpha + 1)}{d_1} + \frac{a(\beta + 1)}{d_2}
$$
The rest is an immediate consequence of the Riemann-Roch formula~\eqref{eq:RR}.
\end{proof}

\section{Rational biruled toric surfaces}
\label{sec:RRTsurfaces-specialtype}

The purpose of this section is to consider the special case $r=0$, since this case has a special 
behavior in terms of cohomology of line bundles. The main result of this section is Theorem~\ref{thm:main}. It states that 
if $r=0$, then all cohomology of line bundles is concentrated in only one degree, all degrees are possible, and 
there are precise formulas for these dimensions as they coincide with the Euler characteristic of the line bundle $\cO(D)$, 
calculated in Lemma~\ref{lemma:chigen}. Now we want to study the global sections $H^0(S,\cO_S(D))$.

Recall from Remark~\ref{rem:special} that any rational biruled toric surface $S_I$ can be characterized by 
two numbers $I=(d,p)$, so that $d=d_1=d_2$, $p=p_1=p_2$, $q=q_1=q_2=d-p$, $r=0$.
Let us consider $D=a Z+b F + \alpha E_X + \beta E_Y\in \Pic(S)$, where $0\leq \alpha,\beta<d$. 
Note that $\hat\rho^*(D)=a Z+b F+\alpha E_X+\beta E_Y +\frac{\alpha}{d} X+\frac{\beta}{d} Y$ and hence
$$\lfloor\hat\rho^*(D)\rfloor=a Z+b F+\alpha E_X+\beta E_Y=(a+b)Z-b E+(\alpha+bp) E_X+(\beta+bq) E_Y\in \Pic(\hat \Sigma_1).$$
On the other hand 
$$\rho^*((a+b)Z)=(a+b)\rho^*(Z)=(a+b)(Z+p E_X+q E_Y)\in \Pic(\hat \Sigma_1).$$
Therefore
$$
\lfloor\hat\rho^*(D)\rfloor=(a+b)\rho^*(Z)-b E-(ap-\alpha)E_X-(aq-\beta)E_Y.
$$
Using the projection formula
\begin{equation}
\label{eq:h0D}
\array{c}
H^0(S,\cO_S(D))
=H^0(\hat\Sigma_1,\cO(\lfloor\hat\rho^*(D)\rfloor))\\ \\
=\rho_*H^0(\hat\Sigma_1,\cO_{\hat\Sigma_1}(\rho^*((a+b)Z)-b E-(ap-\alpha)E_X-(aq-\beta)E_Y))\\ \\
=\left\{ h\in H^0(\PP^2,\cO_{\PP^2}((a+b)Z)) 
\left|
\array{l}
\mult_{(1,1)}h(X,Y,1)\geq b,\\ 
\mult_{(d,q)}h(1,Y,Z)\geq aq-\beta, \\
\mult_{(d,p)}h(X,1,Z)\geq ap-\alpha 
\endarray
\right.
\right\}\\ \\
=\langle X^uY^vZ^{a+b-u-v}\in \CC[X,Y,Z]_{a+b}\mid 
u+v\geq b,\quad -pb-\alpha\leq qu-pv\leq qb+\beta
\rangle
\endarray
\end{equation}

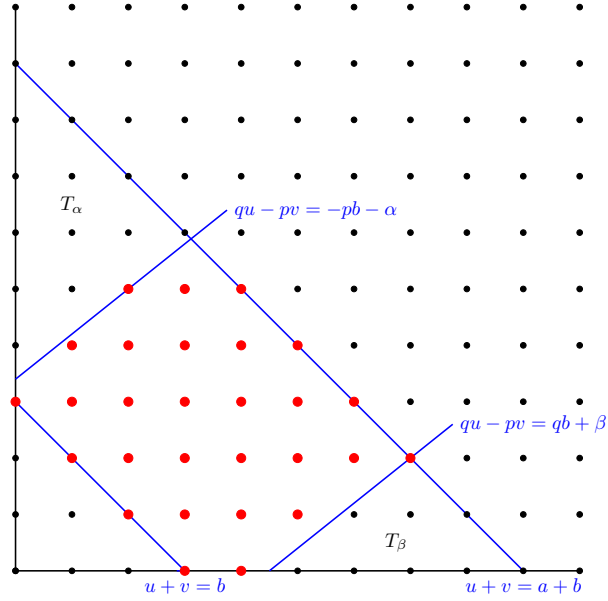
\begin{figure}[ht]
\begin{center}
\scalebox{.75}{
\begin{tikzpicture}
\draw[black,smooth,thick] (0,10) -- (0,0) -- (10,0);
\draw[blue,smooth,thick] (0,9) -- (9,0) node [blue,below] {$u+v=a+b$};
\draw[blue,smooth,thick] (0,3) -- (3,0) node [blue,below] {$u+v=b$};
\draw[blue,smooth,thick] (2+6/7,0) -- (3,0);
\draw[fill] circle [radius=.05] (0,0);
\foreach \x in {0,...,10} 
\foreach \y in {0,...,10} 
    \draw[fill] (\x,\y) circle (0.05cm);

\draw[smooth, color=blue, domain=0:13, thick, line cap=butt, samples=400] plot (\x/4+3+6/4,\x/5) node [blue,right] {$qu-pv=qb+\beta$};
\draw[smooth, color=blue, domain=0:15, thick, line cap=butt, samples=400] plot (\x/4,\x/5+3+2/5) node [blue,right] {$qu-pv=-pb-\alpha$};

\foreach \x in {(0,3),(1,2),(1,3),(1,4),(3,0),(3,1),(3,2),(3,3),(3,4),(3,5),(2,1),(2,2),(2,3),(2,4),(2,5),(4,0),(4,1),(4,2),(4,3),(4,4),(4,5),(5,1),(5,2),(5,3),(5,4),(6,2),(6,3),(7,2)}
    \draw[fill,red] \x circle (0.08cm);
    
\node at (6.75,0.5) {$T_\beta$};
\node at (1,6.5) {$T_\alpha$};
\end{tikzpicture}
}
 \caption{Monomials in $H^0(\cO_S(D))$, $D=6Z+3F+2 E_X+6 E_Y$, $(d,p,q)=(9,5,4)$.}
\label{fig:h0}
\end{center}
\end{figure}

A straightforward calculation --\,see Figure~\ref{fig:h0}\,-- shows that
\begin{equation}
\label{eq:h0D1}
h^0(S,\cO_S(D))=
\begin{cases}
\binom{a+b+2}{2}-\binom{b+1}{2}-\# T_\beta-\# T_\alpha & \text{ if } b\geq 0, a\geq 0\\ 
\binom{a+1}{2}-\# T_\beta-\# T_\alpha & \text{ if } b=-1, a\geq 0, \text{ and } \alpha+\beta\geq d\\
0 & \text{ otherwise.}
\end{cases}
\end{equation}

Let us calculate the number of integer points in the triangles
$T_\beta=\{(u,v)\in \ZZ^2_{\geq 0}\mid qu-pv> qb+\beta, u+v\leq a+b\}$ and 
$T_\alpha=\{(u,v)\in \ZZ^2_{\geq 0}\mid qu-pv< -pb-\alpha, u+v\leq a+b\}$. 
For $T_\beta$ consider the following change of coordinates:
$$
i=qu-pv-qb-\beta-1\geq 0,\quad j=v\geq 0,\quad k=a+b-u-v\geq 0.
$$
Note that $\tilde T_\beta=\{(i,j,k)\in \ZZ^3_{\geq 0}\mid i+dj+qk=qa-\beta-1\}$ is such that
$(i,j,k)\in \tilde T_\beta$ if and only if $(a+b-j-k,j)\in T_\beta$ and 
$(u,v)\in T_\beta$ if and only if $(qu-pv-qb-\beta-1,v,a+b-u-v)\in \tilde T_\beta$.

Then $\# T_\beta=\# \tilde T_\beta=h^0(\PP^2_w,\cO_{\PP^2_w}(qa-\beta-1))$, for $w=(d,1,q)$.
Using the adjunction formula
\begin{equation}
\label{eq:Tbeta}
\array{c}
\chi(\cO_{\PP^2_w}(qa-\beta-1))=g_{w,qa-\beta-1}-\sum_P \Delta_P(|w|+qa-\beta-1)\\
\\
=\frac{(qa-\beta-1)(qa-\beta-1+|w|)}{2\bar w}+1
-\Delta_{\frac{1}{d}(1,q)}(|w|+qa-\beta-1)
-\Delta_{\frac{1}{q}(1,d)}(|w|+qa-\beta-1)\\
\\
=\frac{(qa-\beta-1)(qa-\beta+d+q)}{2dq}+1
-\Delta_{\frac{1}{d}(1,q)}(-\beta+(a+1)q)
-\Delta_{\frac{1}{q}(1,p)}(-\beta+p),
\endarray
\end{equation}
where $|w|=1+d+q$. 
Since
\[
qa-\beta-1-(-|w|)=q(a+1)+(d-\beta)>0
\]
have $h^1(\PP^2_w,\cO_{\PP^2_w}(qa-\beta-1))=h^2(\PP^2_w,\cO_{\PP^2_w}(qa-\beta-1))=0$.
The vanishing of $h^1$ is a general property of weighted-projective spaces --\,see for instance
\cite[p.~39]{Dolgachev-weighted}. The vanishing of $h^2$ follows from Serre's duality since 
$-|w|=-1-q-d<qa-\beta-1$. Therefore, $\chi(\cO_{\PP^2_w}(qa-\beta-1))=h^0(\cO_{\PP^2_w}(qa-\beta-1))$.

Analogously for $T_\alpha$
$$
\array{rcl}
\# T_\alpha&=&\frac{(pa-\alpha-1)(pa-\alpha+d+p)}{2dp}+1
-\Delta_{\frac{1}{d}(1,p)}(-\alpha+(a+1)p)
-\Delta_{\frac{1}{p}(1,q)}(-\alpha+q).
\endarray
$$
and hence~\eqref{eq:h0D1} becomes
\begin{equation}
\label{eq:h0D-final}
\array{rl}
h^0(S,\cO_S(D))=&\!\!\!\binom{a+b+2}{2}\!-\!\binom{b+1}{2}\!-2\!
-\!\frac{(qa-\beta-1)(qa-\beta+d+q)}{2dq}
-\frac{(pa-\alpha-1)(pa-\alpha+d+p)}{2dp}\\ \\
&+\Delta_{\frac{1}{d}(1,q)}(-\beta+(a+1)q)
+\Delta_{\frac{1}{q}(1,p)}(-\beta+p)
\\ \\
&
+\Delta_{\frac{1}{d}(1,p)}(-\alpha+(a+1)p)
+\Delta_{\frac{1}{p}(1,q)}(-\alpha+q).
\endarray
\end{equation}
if $b\geq -1$ and $h^0(S,\cO_S(D))=0$ otherwise.

\begin{lema}
\label{lema:previo}
Consider $D=a Z+b F + \alpha E_X + \beta E_Y\in \Pic(S)$, where $0\leq \alpha,\beta<d$, then
\begin{enumerate}[label=\rm(\arabic*),itemindent=1em]
 \item\label{part1} 
 if either $a\leq -1$ or $b\leq -2$, then $h^0(\cO_S(D))=0$,
 \item\label{part1b}
 if $b=-1$, $a\geq 0$, and $\alpha+\beta<d$, then $h^0(\cO_S(D))=0$,
 \item\label{part2} 
 if either $a\geq -1$, $b\geq 1$, or $a+b\geq -2$, then $h^2(\cO_S(D))=0$,
\end{enumerate}
\end{lema}

\begin{proof}
Part~\ref{part1} is a direct consequence of~\eqref{eq:h0D1}. 
For part~\ref{part2} note that
$K-D\sim a'Z+b'F+\alpha'E_X+\beta'E_Y$, where $a'=-(a+2)$, 
$b'=-(b+1)-d\left(\left[ \frac{\alpha'}{d}\right]+\left[ \frac{\beta'}{d}\right]\right)$,
$\alpha'=d-\alpha-p-1$, and $\beta'=d-\beta-q-1$. In particular 
$-(b+3)\leq b'\leq -(b+1)$. Therefore, if either $a\geq -1$, $b\geq 1$ or $a+b\geq -2$ 
one has $a'\leq -1$, $b'\leq -2$, or $a'+b'\leq -1$. 
Hence, by~\ref{part1} and Serre's Duality this implies $h^2(\cO_S(D))=h^0(\cO_S(K-D))=0$.
\end{proof}

\begin{thm}
\label{thm:main}
Consider $D\sim a Z+b F + \alpha E_X + \beta E_Y\in \Pic(S)$, where $0\leq \alpha,\beta<d$, then
$H^*(S,\cO_S(D))$ is either trivial or concentrated in only one degree $k=k(D)$, for which 
$$h^k(S,\cO_S(D))=(-1)^k\chi(S,\cO_S(D)).$$

Moreover, if $a=-1$, then $\cO_S(D)$ is acyclic. For the remaining cases, $k$ can be 
obtained according to Table{\rm~\ref{tabla}}:

\begin{table}
\begin{center}
\begin{tabular}{c|c|c|}
& $a>-1$ & $a<-1$ \\ \hline &&\\ [-1em]
$\begin{cases} b>-1 \\ b=-1 \text{ and } \alpha+\beta\geq d\end{cases}$ & $k=0$ & $k=1$ \\ [4ex] \hline &&\\ [-1em]
$\begin{cases} b<-1 \\ b=-1 \text{ and } \alpha+\beta< d\end{cases}$ & $k=1$ & $k=2$ \\ [4ex] \hline
\end{tabular}
\end{center}
\caption{}
\label{tabla}
\end{table} 
\end{thm}

\begin{proof}
We shall prove the moreover part which implies the first part of the statement. 
It will be done on a case by case basis. The most involved case being $a,b>-1$.
\begin{enumerate}[label=\rm(\arabic*),itemindent=1em]
 \item\label{part5} Case $a\geq -1$, $b>-1$: 
Note that $h^2(\cO_S(D))=0$ is a consequence of Lemma~\ref{lema:previo}.
Hence it is enough to show $h^0(\cO_S(D))=\chi(\cO_S(D))$ using Riemann-Roch's formula.
From Lemma~\ref{lemma:chigen} for $d_1=d_2=d$, $p_1=p_2=p$, $q_1=q_2=q$, and $r=0$ one has
\begin{equation}
\label{eq:RR2}
\array{rcl}
\chi(\cO_S(D))&=&1+\frac{((a+1)(db+\alpha+\beta+1)-1)}{d}\\ \\
&&-\Delta_{\frac{1}{d}(1,p)}(-\alpha)
-\Delta_{\frac{1}{d}(1,q)}(-\beta)\\ \\
&&-\Delta_{\frac{1}{d}(1,p)}(-\beta-pa)
-\Delta_{\frac{1}{d}(1,q)}(-\alpha-qa)
\endarray
\end{equation}
First we will combine the values of $\Delta_P$, $P\in \Sing(S)$ appearing in~\eqref{eq:h0D-final} 
and~\eqref{eq:RR2}. The following are a consequence of Proposition~\ref{prop:Delta}.
$$
\array{l}
\Delta_{\frac{1}{d}(1,q)}(qa-\beta+q)+\Delta_{\frac{1}{d}(1,p)}(-\beta-pa)=\frac{d-1}{2d}\\ \\
\Delta_{\frac{1}{d}(1,p)}(pa-\alpha+p)+\Delta_{\frac{1}{d}(1,q)}(-\alpha-qa)=\frac{d-1}{2d}.
\endarray
$$
Also, using the adjunction formula
$$
\array{l}
\chi(\cO_{\PP^2_{(1,d,p)}}(-\alpha-1-p))=
g_{(1,d,p),-\alpha-1-p}-\Delta_{\frac{1}{p}(1,q)}(-\alpha+q)-\Delta_{\frac{1}{d}(1,p)}(-\alpha)\\ \\
\chi(\cO_{\PP^2_{(1,d,q)}}(-\beta-1-q))=
g_{(1,d,q),-\beta-1-q}-\Delta_{\frac{1}{q}(1,p)}(-\beta+p)-\Delta_{\frac{1}{d}(1,q)}(-\beta)
\endarray
$$
On the other hand $h^i(\PP^2_{(1,d,p)},\cO_{\PP^2_w}(-\alpha-1-p))=0$, for $i=0$ (since $-\alpha-1-p<0$)
for $i=1$ (general property of quasi-projective planes), and for $i=2$ (using Serre's duality plus the 
fact that $-|w|=-1-d-p<-\alpha-1-p$), then $\chi(\cO_{\PP^2_{(1,d,p)}}(-\alpha-1-p))=0$. Analogously 
$\chi(\cO_{\PP^2_{(1,d,q)}}(-\beta-1-q))=0$. Using the formula for $h^0(\cO_S(D))$ for the case 
$b\geq 0$, $a\geq 0$ one obtains
$$
\array{rl}
h^0(\cO_S(D))-\chi(\cO_S(D))\!&=
\binom{a+b+2}{2}-\binom{b+1}{2}
-\frac{(qa-\beta-1)(qa-\beta+d+q)}{2dq}
-\frac{(pa-\alpha-1)(pa-\alpha+d+p)}{2dp}\\ \\
&-1-\frac{((a+1)(db+\alpha+\beta+1)-1)}{d}
+\frac{d-1}{d}
+\frac{(\alpha+p+1)(\alpha-d)}{2dp}
+\frac{(\beta+q+1)(\beta-d)}{2dq}\!=\!0.
\endarray
$$
The last equality is a straightforward calculation. 

\item\label{part5b} 
Case $b<-1\leq a$: note that $h^2(\cO_S(D))=h^0(\cO_S(D))=0$ by Lemma~\ref{lema:previo}.

\item Case $a\leq -1$, $b<-1$: 
first note that $b< -1$ implies $h^0(\cO_S(D))=0$ by Lemma~\ref{lema:previo}.
Also $a\leq -1$ implies $-(a+2)\geq -1$, and $b< -1$ implies $-(b+1)> -1$.
Hence by Serre's duality $h^1(\cO_S(D))=h^1(\cO_S(K-D))=0$ and case~\ref{part5}.
\item Case $a<-1<b$: again $a\leq -1$ implies $h^0(\cO_S(D))=0$. 
Also $b\geq 0$ implies $h^2(\cO_S(D))=0$ according to case~\ref{part5}.
then $h^1(\cO_S(D))=-\chi(\cO_S(D))$. 
\item Case $b=-1$, $\alpha+\beta\geq d$, and $a\geq -1$: the proof of case~\ref{part5} is 
also valid in this situation.
\item Case $b=-1$, $\alpha+\beta< d$, and $a\geq -1$: according to 
Lemma~\ref{lema:previo}~\ref{part2} $h^2(\cO_S(D))=0$ (since $a\geq -1$). To prove $h^0(\cO_S(D))=0$
we distinguish two cases: if $a\geq 0$ one uses Lemma~\ref{lema:previo}~\ref{part1b} and if
$a=-1$ one can use Lemma~\ref{lema:previo}~\ref{part1}.
\item Case $b=-1$, $\alpha+\beta\geq d$, and $a\leq -1$: 
according to Lemma~\ref{lema:previo}~\ref{part1} $h^0(\cO_S(D))=0$ (since $a\leq -1$).
To prove $h^2(\cO_S(D))=0$ we will use Serre's duality over $K-D\sim a'Z+b'F+\alpha'E_X+\beta'E_Y$
as described in the proof of Lemma~\ref{lema:previo} to show that $h^0(\cO_S(K-D))=0$. 
Note that $a\leq -1$ implies $a'\geq -1$. If $a'=1$, then one can use 
Lemma~\ref{lema:previo}~\ref{part1} to show that $h^0(\cO_S(K-D))=0$. Hence we will assume 
$a'\geq 0$. We will distinguish several cases.
If both $\alpha\geq q$, and $\beta\geq p$, then $b'=-2$ and thus $h^0(\cO_S(K-D))=0$ by 
Lemma~\ref{lema:previo}~\ref{part1}. Otherwise $b'=-1$, which is obtained only if either
$\alpha\geq q$ and $\beta<p$ or $\alpha< q$ and $\beta\geq p$. In either case
$\alpha'+\beta'=2d-\alpha-\beta-2<d$. The result then follows from 
Lemma~\ref{lema:previo}~\ref{part1b}.
\item Case $b=-1$, $\alpha+\beta< d$, and $a\leq -1$: 
according to Lemma~\ref{lema:previo}~\ref{part1} $h^0(\cO_S(D))=0$ (since $a\leq -1$).
To prove $h^1(\cO_S(D))=0$ we will use Serre's duality over $K-D\sim a'Z+b'F+\alpha'E_X+\beta'E_Y$
as described in the proof of Lemma~\ref{lema:previo} to show that $h^1(\cO_S(K-D))=0$.
Note that $b=-1$ and $\alpha+\beta<d$ implies $b'=0$, and hence case~\ref{part5} implies
$h^1(\cO_S(K-D))=0$.
\end{enumerate}

The case $a=-1$, $b\neq -1$ is a consequence of cases~\ref{part5}-\ref{part5b}.
\end{proof}

\begin{obs}
The case $b=-1$ is special and the nullity of $\chi(\cO_S(D))$ 
(which implies acyclicity by Theorem~\ref{thm:main}) depends on the values of $\alpha$ and $\beta$. 
For instance, if $S$ is the rational ruled toric surface given by $(d,p,q,r)=(5,3,2,0)$, 
$D_1\sim Z-F+3E_X+2E_Y$, $D_2\sim Z-F+3E_X+E_Y$, and $D_3\sim Z-F+2E_X+E_Y$, then 
$$
\array{ccc}
h^0(\cO_S(D_1))=\chi(\cO_S(D_1))=1 \quad & h^1(\cO_S(D_1))=0 & h^2(\cO_S(D_1))=0\\
h^0(\cO_S(D_2))=0 & h^1(\cO_S(D_2))=0 & h^2(\cO_S(D_2))=0\\
h^0(\cO_S(D_3))=0 & h^1(\cO_S(D_3))=-\chi(\cO_S(D_3))=1 \quad & h^2(\cO_S(D_3))=0\\
\endarray
$$
\end{obs}

\begin{ejm}
Consider the rational ruled toric surface $S_{(d,p)}$ with sections $Z$ and $E$ and $E^2=Z^2=0$ 
as in Figure~\ref{fig:wnagata}. Note that despite $dE_Y=dE_Y=F$, the divisors $E_X$ and $E_Y$ are not 
equivalent. This can be deduced from the homology calculations, since $D=E_X-E_Y=-F+E_X+(d-1)E_Y$ 
satisfies $\chi(\cO_S(D))=0$ however, $\chi(\cO_S)=1$. Note that in particular the Euler characteristic
of $\cO_S(D)$ is sensitive to the torsion of $\Pic(S)$. In the Riemann-Roch formula~\eqref{eq:RR} the 
first summand $\frac{1}{2}D(D+Z_K)$ is only numerical, however $\Delta_S$ is not. In particular 
$\Delta_S(-D)=1$ but $\Delta_S(dD)=\Delta_S(0)=0$.
\end{ejm}

\section{Rational uniruled toric surfaces}
\label{sec:RRTsurfaces-generaltype}

Let $S$ be a toric ruled surface, 
where a section has two quotient singular points $\frac{1}{d_i}(1,n_i)$, $i=1,2$, with $\gcd(d_i,n_i)=1$, $0<n_i\leq d_i$, 
and self-intersection~$r>0$ such that $\frac{n_1}{d_1}+\frac{n_2}{d_2}-r=1-k\in \ZZ$ ($k\geq 0$). 
Since the value of $n_i$ is only well defined modulo $d_i$, one can consider $q_2:=n_2\geq n_1$, $p_1=n_1+kd_1$,
$p_2=d_2-q_2$, and $q_1=d_1-n_1$, in which case
$$
\frac{p_1}{d_1}+\frac{q_2}{d_2}-r=1.
$$
Hence,
\begin{equation}
\label{eq:pq1} 
d_2 p_1+d_1 q_2=d_1 d_2+rd_1d_2>d_1d_2.
\end{equation}

Similarly to the discussion in section~\ref{sec:RRTsurfaces-specialtype} note that 
$H^0(S,\cO_S(D))$ is isomorphic to the linear subspace $V_0$ with basis
\[
\{X^u Y^v\mid u,v\geq 0,\ b\leq u+v\leq a+b,\ q_2 u-p_2v\leq p_2 b+\beta,\ q_1 u-p_1v\geq -p_1 b-\alpha\}.
\]
Let us consider necessary conditions for $V_0$ to be different from~$\{0\}$.
It is easily seen that both $a+b\geq -1$ and $a\geq -1$ are required.
The last two conditions are determined by the lines
$\ell_1=\{p_1 v-q_1u= p_1 b+\alpha\}$ and $\ell_2=\{q_2 u-p_2v= q_2 b+\beta\}$; 
their intersection point is in the line 
\[
u+v=b+\frac{b+\frac{\alpha}{d_1}+\frac{\beta}{d_2}}{r}.
\]
In particular, we can also assume $b\geq -1$. This point is
\[
\left(\frac{ \frac{p_{1}}{d_1} b +  \frac{\alpha}{d_1} }{r}+\frac{p_{1} \beta- q_{2} \alpha}{d_{1} d_{2} r},\,\frac{ 
\frac{q_2}{d_2} b + \frac{\beta}{d_2}}{ r}-\frac{p_{1} \beta- q_{2} \alpha}{d_{1} d_{2} r}\right).
\]
Note that in general, $\ell_1$ and $\ell_2$ are not parallel anymore, since their slopes are $m(\ell_i)=\frac{q_i}{p_i}$ 
and $q_2p_1-q_1p_2=q_2p_1-(d_1-p_1)(d_2-q_2)=p_1d_2+q_2d_1-d_1d_2>0$ by~\eqref{eq:pq1}.

Two cases will be distinguished:

\subsection{Case \texorpdfstring{$k=0$}{k=0}}
In this case one can check that both slopes $m(\ell_i)$ are positive. 
Hence for values $a<\frac{b+\frac{\alpha}{d_1}+\frac{\beta}{d_2}}{r}$ --\,see Figure~\eqref{fig:h01a}\,--
note that $h^0(\cO_S(D))$ depends on $a$, however if $a\geq \frac{b+\frac{\alpha}{d_1}+\frac{\beta}{d_2}}{r}$ 
--\,see Figure~\eqref{fig:h01b}\,--, then $h^0(\cO_S(D))$ is independent of $a$.

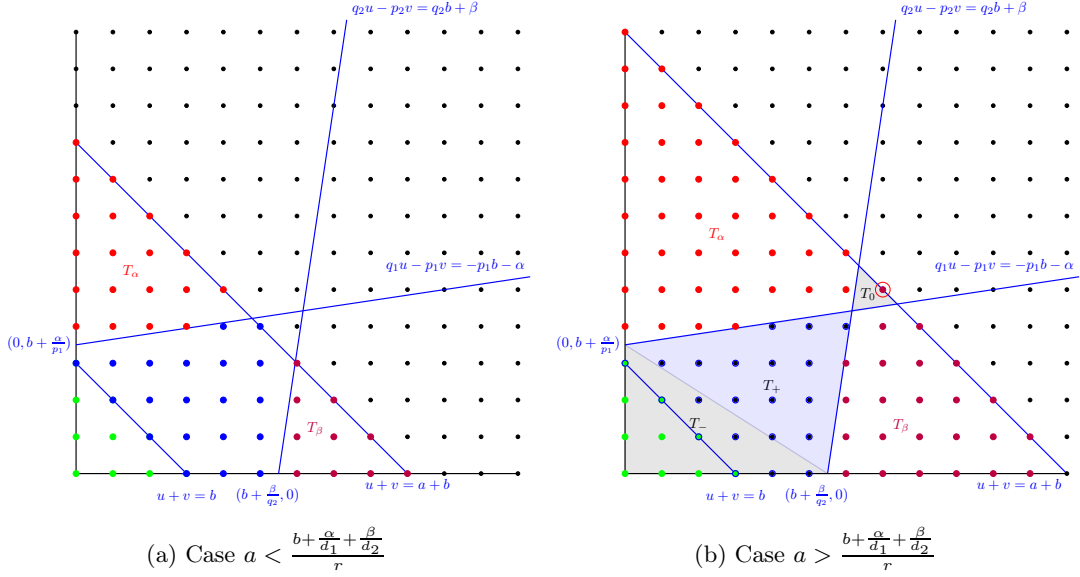
\begin{figure}
\centering
\begin{subfigure}[b]{0.49\textwidth}
    \centering
\scalebox{.55}{
\begin{tikzpicture}[scale=.89]
\draw[black,smooth,thick] (0,12) -- (0,0) -- (12,0);
\draw[blue,smooth,thick] (0,9) -- (9,0) node [blue,below] {$u+v=a+b$};
\draw[blue,smooth,thick] (0,3) -- (3,0) node [blue,below]  at (3,-.3) {$u+v=b$};
\draw[blue,smooth,thick] (2+6/7,0) -- (3,0);
\draw[fill] circle [radius=.05] (0,0);
\foreach \x in {0,...,12} 
\foreach \y in {0,...,12} 
    \draw[fill] (\x,\y) circle (0.05cm);

\draw[smooth, color=blue, domain=0:37, thick, line cap=butt, samples=400] plot (\x/20+4+6/4,\x/3) node [blue,above right] {$q_2u-p_2v=q_2b+\beta$};
\draw[smooth, color=blue, domain=0:37, thick, line cap=butt, samples=400] plot (\x/3,\x/20+3.5) node [blue,above left] {$q_1u-p_1v=-p_1b-\alpha$};

\foreach \x in {(0,4),(0,5),(0,6),(0,7),(0,8),(0,9),(1,4),(1,5),(1,6),(1,7),(1,8),(2,4),(2,5),(2,6),(2,7),(3,4),(3,5),(3,6),(4,5)} 
\draw[fill,red] \x circle (0.08cm);
\node[color=blue] at (-1,3.5) {$(0,b+\frac{\alpha}{p_1})$};
\node[color=red] at (1.5,5.5) {$T_\alpha$};

\foreach \x in {(6,0),(6,1),(6,2),(6,3),(7,0),(7,1),(7,2),(8,0),(8,1),(9,0)} 
\draw[fill,purple] \x circle (0.08cm);
\node[color=blue] at (5.2,-0.6) {$(b+\frac{\beta}{q_2},0)$};
\node[color=purple] at (6.5,1.2) {$T_\beta$};

\foreach \x in {(0,0),(0,1),(0,2),(0,3),(1,0),(1,1),(1,2),(2,0),(2,1),(3,0)} 
\draw[fill,green] \x circle (0.08cm);

\foreach \x in {(0,3),(1,2),(1,3),(2,1),(2,2),(2,3),(3,0),(3,1),(3,2),(3,3),(4,0),(4,1),(4,2),(4,3),(4,4),(5,0),(5,1),(5,2),(5,3),(5,4)} 
\draw[fill,blue] \x circle (0.08cm);

\end{tikzpicture}
}
\caption{Case $a<\frac{b+\frac{\alpha}{d_1}+\frac{\beta}{d_2}}{r}$}
\label{fig:h01a}
\end{subfigure}%
~ 
\begin{subfigure}[b]{0.49\textwidth}
   \centering
\scalebox{.55}{
\begin{tikzpicture}[scale=.89]
\draw[black,smooth,thick] (0,12) -- (0,0) -- (12,0);
\draw[blue,smooth,thick] (0,12) -- (12,0) node [blue,below left] {$u+v=a+b$};
\draw[blue,smooth,thick] (0,3) -- (3,0) node [blue,below]  at (3,-.3) {$u+v=b$};
\draw[blue,smooth,thick] (2+6/7,0) -- (3,0);
\draw[fill] circle [radius=.05] (0,0);
\draw[black,fill,opacity=0.1,smooth,thick] (0,3.5) -- (5.5,0) -- (0,0) -- cycle node [black,below,opacity=1] at (4,2.7) {$T_+$};
\draw[blue,fill,opacity=0.1,smooth,thick] (0,3.5) -- (5.5,0) -- (6.15,4.4) -- cycle node [black,below,opacity=1] at (2,1.66) {$T_-$};
\draw[black,fill,opacity=0.1,smooth,thick] (7.4,4.6) -- (6.35,5.65) -- (6.15,4.4) -- cycle node [black,below,opacity=1] at (6.6,5.2) {$T_0$};
\foreach \x in {0,...,12} 
\foreach \y in {0,...,12} 
    \draw[fill] (\x,\y) circle (0.05cm);

\draw[smooth, color=blue, domain=0:37, thick, line cap=butt, samples=400] plot (\x/20+4+6/4,\x/3) node [blue,above right] {$q_2u-p_2v=q_2b+\beta$};
\draw[smooth, color=blue, domain=0:37, thick, line cap=butt, samples=400] plot (\x/3,\x/20+3.5) node [blue,above left] {$q_1u-p_1v=-p_1b-\alpha$};

\foreach \x in {(0,4),(0,5),(0,6),(0,7),(0,8),(0,9),(0,10),(0,11),(0,12),(1,4),(1,5),(1,6),(1,7),(1,8),(1,9),(1,10),(1,11),(2,4),(2,5),(2,6),(2,7),(2,8),(2,9),(2,10),(3,4),(3,5),(3,6),(3,7),(3,8),(3,9),(4,5),(4,6),(4,7),(4,8),(5,5),(5,6),(5,7),(6,5),(6,6)} 
\draw[fill,red] \x circle (0.08cm);
\draw[red] (7,5) circle (0.2cm);
\node[color=blue] at (-1,3.5) {$(0,b+\frac{\alpha}{p_1})$};
\node[color=red] at (2.5,6.5) {$T_\alpha$};

\foreach \x in {(6,0),(6,1),(6,2),(6,3),(7,0),(7,1),(7,2),(7,3),(7,4),(8,0),(8,1),(8,2),(8,3),(8,4),(9,0),(9,1),(9,2),(9,3),(10,0),(10,1),(10,2),(11,0),(11,1)}
\draw[fill,purple] \x circle (0.08cm);
\draw[fill,purple] (7,5) circle (0.08cm);
\node[color=blue] at (5.2,-0.6) {$(b+\frac{\beta}{q_2},0)$};
\node[color=purple] at (7.5,1.3) {$T_\beta$};

\foreach \x in {(0,0),(0,1),(0,2),(0,3),(1,0),(1,1),(1,2),(2,0),(2,1),(3,0)} 
\draw[fill,green] \x circle (0.08cm);

\foreach \x in {(0,3),(1,2),(1,3),(2,1),(2,2),(2,3),(3,0),(3,1),(3,2),(3,3),(4,0),(4,1),(4,2),(4,3),(4,4),(5,0),(5,1),(5,2),(5,3),(5,4),(6,4)} 
\draw[blue,thick] \x circle (0.08cm);

\end{tikzpicture}
}
\caption{Case $a>\frac{b+\frac{\alpha}{d_1}+\frac{\beta}{d_2}}{r}$}
\label{fig:h01b}
\end{subfigure}
\caption{Monomials in $H^0(\cO_S(D))$ case $k=0$.}
\label{fig:h01}
\end{figure}

\begin{thm}
\label{thm:main2}
Assume $D\sim aZ+bF+\alpha E_X+\beta E_Y$, where $a\geq 0$ and $b\geq -1$, then 
\begin{itemize}
 \item if $\frac{b+\frac{\alpha}{d_1}+\frac{\beta}{d_2}}{r}>a\geq 0$, $b\geq -1$, then 
 $$\array{cc} h^0(\cO_S(D))=\chi(\cO_S(D)), & h^1(\cO_S(D))=h^2(\cO_S(D))=0\endarray$$ 
 \item otherwise $\chi(\cO_S(D))$ is given as in Lemma~\ref{lemma:chigen}, $h^2(\cO_S(D))=0$ and 
 $$h^1(\cO_S(D))=h^0(\cO_S(D))-\chi(\cO_S(D))=h^0(\PP^2_w;\cO(\rho)),$$ 
 where $w=(d_1,d_2,rd_1d_2)$ and $\rho=rd_1d_2a-d_1d_2b-d_1\beta-d_2\alpha\geq 0$ is quadratic in~$a$.
\end{itemize}
\end{thm}

\begin{proof}
For the first part one can apply the following formula 
\begin{equation}
\label{eq:h0gen}
h^0(\cO_S(D))=\binom{a+b+2}{2}-\binom{b+1}{2}-\# T_\alpha - \# T_\beta
\end{equation}
--\,in Figure~\eqref{fig:h01a} the dimension is given by the solid blue dots.
An analogous proof to that of Theorem~\ref{thm:main} case~\ref{part5} using~\eqref{eq:h0gen} and Lemma~\ref{lemma:chigen} shows that 
$$h^0(\cO_S(D))-\chi(\cO_S(D))=\frac{(d_1d_2r+d_1d_2-d_2p_1-d_1q_2)(a+1)a}{2d_1d_2},$$
which is zero by~\eqref{eq:pq1}. The result follows direct calculation of $h^0(\cO_S(K-D))=h^2(\cO_S(D))=0$ (by Serre's duality).

For the second part one has to apply a different formula for $h^0(\cO_S(D))$ which is independent of $a$, namely 
$$
h^0(\cO_S(D))=T_+ + T_- - \binom{b+1}{2}
$$
--\,in Figure~\eqref{fig:h01b} the dimension is given by the solid blue dots. However, note that the formula for $\chi(\cO_S(D))$ 
given as in Lemma~\ref{lemma:chigen} is quadratic in $a$. An alternative way to see this is using the triangles $T_\alpha$ and 
$T_\beta$ from the previous paragraph. The formula $\chi(\cO_S(D))=\binom{a+b+2}{2}-\binom{b+1}{2}-\# T_\alpha - \# T_\beta$ is 
still true, however the right-hand side now equals $h^0(\cO_S(D))-\# T_0$. Since $h^2(\cO_S(D))=0$ as above, one obtains 
$h^1(\cO_S(D))=\# T_0$. Using the techniques introduced in Section~\ref{sec:RRTsurfaces-specialtype} one can rewrite the number of 
integer points in a triangle as the dimension of the space of global sections of some weighted projective space. 
In this case one can check that
$$
\array{rcl}
\# T_0 & = & h^0(\PP^2_w, \cO(\rho))\\
&=&g_{w,\rho}-\Delta_{\frac{1}{d_1}(d_2,rd_1d_2)}(d_2\alpha-rd_1d_2a)
-\Delta_{\frac{1}{d_2}(d_1,rd_1d_2)}(d_1\beta-rd_1d_2a)\\
&&-\Delta_{\frac{1}{rd_1d_2}(d_1,d_2)}(d_1\beta+d_2\alpha-d_1d_2b-rd_1d_2a),
\endarray
$$
where $w=(d_1,d_2,rd_1d_2)$ and $\rho=rd_1d_2a-d_1d_2b-d_1\beta-d_2\alpha\geq 0$ by hypothesis.
\end{proof}

\subsection{Case \texorpdfstring{$k>0$}{k>0}}
In this case one can check that $-1<m(\ell_1)<0$ and $m(\ell_2)>0$. Hence again, for 
values $a<\frac{b+\frac{\alpha}{d_1}+\frac{\beta}{d_2}}{r}$ --\,see Figure~\eqref{fig:h01a}\,--
note that $h^0(\cO_S(D))$ depends on $a$, however if $a\geq \frac{b+\frac{\alpha}{d_1}+\frac{\beta}{d_2}}{r}$ 
--\,see Figure~\eqref{fig:h01b}\,--, then $h^0(\cO_S(D))$ is independent of $a$.

The case $a\geq \frac{b+\frac{\alpha}{d_1}+\frac{\beta}{d_2}}{r}$ however has a new situation that we cover in this result.

\begin{thm}
\label{thm:h02}
Assume $D\sim aZ+bF+\alpha E_X+\beta E_Y$, where $b\geq -1$,
$a\geq \frac{b+\frac{\alpha}{d_1}+\frac{\beta}{d_2}}{r}\geq 0$, and $-q_1a\geq d_1b+\alpha$, then
$$
h^0(\cO_S(D))=h^0(\PP^2_{w_1}, \cO(\rho_1))-h^0(\PP^2_{w_2}, \cO(\rho_2))-\binom{b+1}{2},
$$
where $w_1=(1,p_1,-q_1)$, $\rho_1=bp_1+\alpha$, $w_2=(-q_1,q_2,rd_1d_2)$, and $\rho_2=q_2d_1b+\alpha q_2+\beta q_1$. 
\end{thm}

\section{Cyclic branched coverings on rational ruled toric surfaces}\label{sec:esnault}

Let $S$ be a rational ruled toric surface as in Definition~\ref{def:trrs} and consider
$\pi: \wt{S} \to S$ a cyclic branched covering of $n$ sheets. Assume the ramification set
is given by a $\QQ$-normal crossing Weil divisor $D=\sum_{i=1}^r m_i D_i$ linearly equivalent to $nH$
for some Weil divisor $H$. One is interested in studying the monodromy of the covering acting
on $H^1(\wt{S},\CC) = H^0(\wt{S},\Omega^1_{\wt{S}}) \oplus H^1(\wt{S},\cO_{\wt{S}})$.
Since both summands are complex conjugated
the decomposition becomes
$$
H^1(\wt{S},\CC) = H^1(\wt{S},\cO_{\wt{S}}) \oplus \overline{H^1(\wt{S},\cO_{\wt{S}})}.
$$

In the smooth case, Esnault and Viehweg~\cite{Esnault-Viehweg82} proved that
$H^1(\wt{S},\cO_{\wt{S}}) = H^1(S,\pi_{*}\cO_{\wt{S}})$ and provided a precise
description of the sheaf $\pi_{*}\cO_{\wt{S}}$ giving its Hodge structure
compatible with the monodromy of the covering. In a forthcoming paper we will
prove that the same is true for surfaces with abelian quotient singularities.
In particular, the following result holds.

\begin{thm}\label{thm:Esnault}
Using the previous notation,
$$
  H^1(\wt{S},\cO_{\wt{S}}) = \bigoplus_{k=0}^{n-1} H^1 (S, \cO_S(L^{(k)})), \qquad
  L^{(k)} = -k H + \sum_{i=1}^r \floor{\frac{km_i}{n}} D_i,
$$
where the monodromy of the cyclic covering acts on $H^q(S,\cO_S(L^{(k)}))$ by multiplication by~$e^{\frac{2\pi ik}{n}}$.
\end{thm}

\begin{obs}
Note that any covering can be encoded by the data $(S,D,H,n)$ such that $D,H \in \Pic(S)$ and $D \sim nH$.
If $\Pic(S)$ is torsion free, e.g.~if $S$ is smooth and rational, then the divisor $H$ is uniquely determined by $(S,D,n)$. 
Otherwise, this is no longer true, both if $S$ is rational --\,see Example~\ref{ex:covering}\,-- or even if $S$ is smooth --\,for instance
an Enriques surface $S$ where $D=0$, $n=2$, and $H$ can be chosen to be either $H=0$ or $H=K_S$ the canonical divisor.
\end{obs}

\subsection{The divisor \texorpdfstring{$L^{(k)}$}{Lk}} Let us white $D$ in terms of its irreducible components
$D = \sum_{i=1}^r m_i D_i$ and suppose $D_i \sim a_i Z + \alpha_i E_X + \beta_i E_Y$ and $H \sim z Z + e_x E_X + e_y E_Y$
for some $a_i,\alpha_i,\beta_i,z,e_x,e_y \in \mathbb{Z}$. Recall $F \cdot E_X = F \cdot E_Y = 0$ and $F \cdot Z = 1$.
Assume without loss of generality that $D$ is an effective divisor.
Since $D$ is linearly equivalent to $nH$, $D \cdot F = nH \cdot F$ and hence $\sum_{i=1}^r m_i a_i = nz$.
Analogously, using $D \cdot Z = nH \cdot Z$ and the fact that $Z^2=0$, $Z \cdot E_X = Z \cdot E_Y = \frac{1}{d}$,
one obtains $\sum_{i=1}^r m_i(\alpha_i + \beta_i) = n(e_x+e_y)$.

The divisor $L^{(k)} = -kH + \sum_{i=1}^r \floor{\frac{km_i}{n}} D_i \in \Pic(S)$ can be
rewritten as $u_k Z + v_k E_X + w_k E_Y$, where the coefficients are given by
$$
u_k = -kz+\sum_{i=1}^r \floor{\frac{km_i}{n}}a_i, \quad
v_k = -ke_x + \sum_{i=1}^r \floor{\frac{km_i}{n}}\alpha_i, \quad
w_k = -ke_y + \sum_{i=1}^r \floor{\frac{km_i}{n}}\beta_i.
$$

Note that
\begin{equation}\label{eq:uvw}
u_k = \sum_{i=1}^r \left( -\frac{km_i}{n} + \floor{\frac{km_i}{n}} \right) a_i, \qquad
v_k + w_k = \sum_{i=1}^r \left( -\frac{km_i}{n} + \floor{\frac{km_i}{n}} \right) (\alpha_i + \beta_i).
\end{equation}
Since all $D_i$'s are effective divisors, $a_i = D_i \cdot F \geq 0$ and $\alpha_i+\beta_i = D_i \cdot (dZ) \geq 0$.
In particular, $u_k \leq 0$ and $v_k+w_k\leq 0$ for all $k = 0,\ldots,n-1$.

\subsection{The first cohomology group of \texorpdfstring{$\cO_S(L^{(k)})$}{OSLk}}
According to Theorem~\ref{thm:Esnault}, one is interested in computing $h^1(S,\cO_S(L^{(k)}))$.
To do so we will apply Theorem~\ref{thm:main} to the divisor~$L^{(k)}$.
First one needs to divide $v_k$ and $w_k$ by $d$, that is, $v_k = c_{k,1} d + r_{k,1}$ and $w_k=c_{k,2}d+r_{k,2}$,
so that $L^{(k)} \sim u_k Z + (c_{k,1}+c_{k,2}) F + r_{k,1} E_X + r_{k,2} E_Y$.
Four different cases are discussed.
\begin{enumerate}
\item $u_k < 0$ and $v_k+w_k<0$. Then $c_{k,1}+c_{k,2} = \frac{v_k+w_k}{d}-\frac{r_{k,1}+r_{k,2}}{d} < 0$.
If either $u_k=-1$; $u_k < -1$ and $c_{k,1}+c_{k,2}<-1$; or $u_k < -1$, $c_{k,1}+c_{k,2}=-1$, and $r_{k,1}+r_{k,2} < d$,
then $h^1(S,\cO_S(L^{(k)}))=0$. Note that the case $u_k < -1$, $c_{k,1}+c_{k,2}=-1$, and $r_{k,1}+r_{k,2} \geq d$
cannot occur.
\item $u_k = 0$ and $v_k+w_k = 0$. Then $c_{k,1}+c_{k,2} = - \frac{r_{k,1}+r_{k,2}}{d}$ can be either zero or minus one.
In both cases, the first cohomology group vanishes.
\item $u_k = 0$ and $v_k+w_k < 0$. Then as above $c_{k,1}+c_{k,2}<0$.
The case $c_{k,1}+c_{k,2}=-1$ and $r_{k,1}+r_{k,2} \geq d$ is not compatible with $v_k+w_k<0$.
Thus $h^1(S,\cO_S(L^{(k)})) = -\chi(S,\cO_S(L^{(k)})) = -1 - \floor{\frac{v_k}{d}} - \floor{\frac{w_k}{d}}$.
The latter formula holds by Lemma~\ref{one-direction-euler}, see below. In particular, if $v_k+w_k=-1$,
then the dimension is zero.
\item $u_k < 0$ and $v_k+w_k = 0$. The sheaf $L^{(k)}$ is acyclic for $u_{k}=-1$. If $u_{k}<-1$, then
$h^1(S,\cO_S(L^{(k)})) = -\chi(S,\cO_S(L^{(k)})) = -1 - \floor{\frac{u_k-v_kp'}{d}} - \floor{\frac{v_kp'}{d}}$,
where $p'p \equiv 1 \! \mod d$, see Lemma~\ref{one-direction-euler}.
Note that the previous formula still holds for $u_k=-1$, since in such a case it vanishes.
\end{enumerate}

\begin{lema}\label{one-direction-euler}
$\chi(S,\cO_S(v_k E_X + w_k E_Y)) = 1 + \floor{\frac{v_k}{d}} + \floor{\frac{w_k}{d}}$ and
$\chi(S,\cO_S(u_k Z + v_k (E_X - E_Y))) = 1 + \floor{\frac{u_k-v_k p'}{d}} + \floor{\frac{v_k p'}{d}}$.
\end{lema}

\begin{proof}
As above, using the divisions $v_k = c_{k,1} d + r_{k,1}$ and $w_k=c_{k,2}d+r_{k,2}$, one writes
the divisor $v_k E_X + w_k E_Y$ as $(c_{k,1}+c_{k,2})F + r_{k,1} E_X + r_{k,2} E_Y$.
By Lemma~\ref{lemma:chigen},
$$
\begin{aligned}
\chi(S,\cO_S(v_k E_X + w_k E_Y)) &= 1 + c_{k,1} + c_{k,2} + \frac{r_{k,1}}{d} + \frac{r_{k,2}}{d}
- \Delta_{\frac{1}{d}(1,q)}(-r_{k,1}) - \Delta_{\frac{1}{d}(1,p)}(-r_{k,2}) \\
& \quad - \Delta_{\frac{1}{d}(1,p)}(-r_{k,1}) - \Delta_{\frac{1}{d}(1,q)}(-r_{k,2}).
\end{aligned}
$$
Note that by definition $c_{k,1} = \floor{\frac{v_{k}}{d}}$ and $c_{k,2} = \floor{\frac{w_{k}}{d}}$.
The first part of the statement follows since, by Proposition~\ref{prop:Delta},
$\Delta_{\frac{1}{d}(1,p)}(-r_{k,i}) + \Delta_{\frac{1}{d}(1,q)}(-r_{k,i}) = \frac{r_{k,i}}{d}$, $i=1,2$.
For the second part, one uses the relation calculated in~\eqref{eq:Picard_S} to rewrite the divisor
$u_k Z + v_k(E_X-E_Y) \in \Pic(S)$ as $(u_k-v_k p') Z + v_kp' E$. The symmetry of the surface $S$
completes the proof.
\end{proof}

The previous discussion can be summarized in the following result.

\begin{prop}\label{eq:h1Lk}
The first cohomology group of the sheaf $\cO_S(L^{(k)})$, being $L^{(k)} = u_k Z + v_k E_X + w_k E_Y$, is
\begin{equation*}
h^1(S,\cO_S(L^{(k)})) = 
\begin{cases}
-1 - \floor{\frac{v_k}{d}} - \floor{\frac{w_k}{d}} & u_k=0, \ v_k+w_k \leq -2, \\
-1 - \floor{\frac{u_k-v_kp'}{d}} - \floor{\frac{v_kp'}{d}} & v_k+w_k=0, \ u_k \leq -2, \\
0 & \text{otherwise}.
\end{cases}
\end{equation*}
\end{prop}

\subsection{Decomposition of \texorpdfstring{$H^1(\wt{S},\cO_{\wt{S}})$}{H1OStilde}}
Note that $u_k = 0$ happens only when $\frac{km_i}{n} \in \ZZ$ for all $i \in I_1:= \{ i \mid a_i \neq 0 \}$,
see~\eqref{eq:uvw}, or equivalently, when $\frac{n}{\gcd(n,\{m_i\}_{i\in I})}$ divides $k$.
Let us denote by $n_1 = \gcd(n,\{m_i\}_{i\in I_1})$ and $k=\frac{n}{n_1}k_1$, where $k_1$ runs from $0$ to $n_1-1$.
Analogously, consider $I_2:= \{ i \mid \alpha_i + \beta_i \neq 0 \}$, then $v_k+w_k \neq 0$ holds if and only if
$\frac{n}{\gcd(n,\{m_i\}_{i\in I_2})}$ divides $k$ -- denote by $n_2 = \gcd(n,\{m_i\}_{i\in I_2})$, $k=\frac{n}{n_2}k_2$,
$k_2 = 0,\ldots, n_2-1$.
The direct sum in Theorem~\ref{thm:Esnault} splits into two parts as follows,
\begin{equation}\label{eq:decomp_H1}
  H^1(\wt{S},\cO_{\wt{S}}) = \bigoplus_{k_1=0}^{n_1-1} H^1 (S, \cO_S(L^{(\frac{nk_1}{n_1})}))
  \oplus \bigoplus_{k_2=0}^{n_2-1} H^1 (S, \cO_S(L^{(\frac{nk_2}{n_2})})).
\end{equation}

Let $\pi_j: \wt{S}_j \to S$, $j=1,2$, be the branched covering of $n_j$ sheets whose ramification set is
given by $D=\sum_{i=1} m_i D_i$ with associated divisor $H_j=\frac{n}{n_j}H$. A simple observation
shows that the first (resp.~second) direct sum in~\eqref{eq:decomp_H1} coincides with $H^1 (\wt{S}_1,\cO_{\wt{S}_1})$
(resp.~$H^1 (\wt{S}_2,\cO_{\wt{S}_2})$) and the decomposition is compatible with the monodromy
of $\pi_1$ (resp.~$\pi_2$), since $e^{\frac{2\pi ik}{n}} = e^{\frac{2 \pi i k_1}{n_1}}$
(resp.~$e^{\frac{2\pi ik}{n}} = e^{\frac{2 \pi i k_2}{n_2}}$).

We have just proven the next result.

\begin{thm}
The characteristic polynomial of the monodromy of $\pi:\wt{S} \to S$, $\pi_1:\wt{S}_1 \to S$,
and $\pi_2:\wt{S}_2\to S$ satisfy
$$
p_{H^1 (\wt{S},\cO_{\wt{S}})}(t)
=p_{H^1 (\wt{S}_1,\cO_{\wt{S}_1})}(t) \cdot p_{H^1 (\wt{S}_2,\cO_{\wt{S}_2})}(t),
$$
and moreover the equality can be understood at the level of its Hodge structure.
Using Proposition~\eqref{eq:h1Lk}, explicit formulas for the characteristic polynomials
$p_{H^1 (\wt{S}_j,\cO_{\wt{S}_j})}(t)$ can be calculated.
\end{thm}

Some examples will appear in the following section.

\section{Examples}

\begin{ejm}\label{ex:covering12}
Denote by $\widetilde S_i$ the covering of $S_I$ with $I=(d_1,d_2,p_1,q_2,r)$, where $d_1=d_2=12,p_1=1,q_2=11,r=0$, given by the data 
$n=12$ and $D=F\sim 12H_i$ for $H_i=E_X+i(E_X-E_Y)$, $i=0,\dots,11$. Due to the symmetry of the surface $S_I$, the covers $\widetilde S_i$ 
and $\widetilde S_{11-i}$ are isomorphic. 
In these cases $L_i^{(k)}=-k(E_X+i(E_X-E_Y))$ and $h^1(L_i^{(k)})=-1-\floor{\frac{-k-ki}{12}}-\floor{\frac{ki}{12}}$.
A straightforward calculation shows 

\begin{center}
\begin{tabular}{c|c|c|c|c|c|c|c|c|c|c|c|}
$h^1(L_i^{(k)})$ & $k=1$ & $2$ & $3$ & $4$& $5$ & $6$ & $7$ & $8$ & $9$& $10$ & $11$\\ \hline
$i=0$ & $0$ & $0$ & $0$ & $0$ & $0$ & $0$ & $0$ & $0$ & $0$ & $0$ & $0$\\ \hline
$i=1$ & $0$ & $0$ & $0$ & $0$ & $0$ & $0$ & $0$ & $1$ & $1$ & $1$ & $1$\\ \hline
$i=2$ & $0$ & $0$ & $0$ & $0$ & $1$ & $0$ & $0$ & $0$ & $1$ & $1$ & $1$\\ \hline
$i=3$ & $0$ & $0$ & $0$ & $0$ & $0$ & $0$ & $1$ & $0$ & $0$ & $1$ & $1$\\ \hline
$i=4$ & $0$ & $0$ & $0$ & $0$ & $1$ & $0$ & $0$ & $1$ & $0$ & $1$ & $1$\\ \hline
$i=5$ & $0$ & $0$ & $0$ & $0$ & $0$ & $0$ & $1$ & $0$ & $1$ & $0$ & $1$\\ \hline
\end{tabular}
\end{center}

\vspace*{14pt}
Using formula~\eqref{eq:decomp_H1} this shows that $h^1(\widetilde S_0)=0$, 
$h^1(\tilde S_3)=h^1(\widetilde S_5)=6$, and $h^1(\tilde S_2)=h^1(\widetilde S_4)=8$.
This already distinguishes the cases $i=0$, $\{3,5\}$ and $\{2,4\}$. In the remaining cases
their characteristic polynomial of the monodromy are different and hence the covers
$\tilde S_3\to S_3$ and $\tilde S_5\to S_5$ (resp. $\tilde S_2\to S_2$ and $\tilde S_4\to S_4$) 
can be distinguished.
\end{ejm}

\begin{ejm}\label{ex:covering}
Analogously to Example~\ref{ex:covering12} for the case $d_1=d_2=5,p_1=1,q_2=4,r=0$. 
A straightforward calculation shows 

\begin{center}
\begin{tabular}{c|c|c|c|c|c|c|}
$h^1(L_i^{(k)})$ & $k=0$ & $k=1$ & $k=2$ & $k=3$ & $k=4$\\ \hline
$i=0$ & $0$ & $0$ & $0$ & $0$ & $0$\\ \hline
$i=1$ & $0$ & $0$ & $0$ & $1$ & $1$\\ \hline
$i=2$ & $0$ & $0$ & $1$ & $0$ & $1$\\ \hline
\end{tabular}
\end{center}
\vspace*{14pt}
Again, using formula~\eqref{eq:decomp_H1}, $h^1(\widetilde S_0)=0$, $h^1(\tilde S_1)=h^1(\widetilde S_2)=4$.
This already distinguishes the cases $i=0$ and $i\neq 0$. However, note that the different eigenspaces of the monodromy action 
as described in Theorem~\ref{thm:Esnault} distinguishes the covers $\widetilde S_i\to S_i$ for the cases $i\in \{0,1,2\}$.
\end{ejm}

\begin{ejm}
We consider a surface $S$ obtained as follows. Start with $\bp^1\times\bp^1$ with two $0$-sections $S_0,S_\infty$
and four fibers $F_j$, $1\leq j\leq 4$. We perform four weighted Nagata transformations; we start with $(1,2)$-blow-ups
over $S_0\cap F_j$, $j=1,2$, and $S_\infty\cap F_j$, $j=3,4$. Blowing-down the strict transforms of the fibers,
we obtain a new surface with $8$ points of type $\frac{1}{2}(1,1)$ in the intersection points of the strict transforms
of $S_0,S_\infty$ (we keep their notation) with the four new fibers $\tilde{F}_j$, $1\leq j\leq 4$. Note
that $S_0^2=S_\infty^2=0$. Using Propositions~\ref{prop:pic1} and~\ref{prop:pic2} we obtain
\[
\Pic(S)=\langle S_0,F,\tilde{F}_1,\dots,\tilde{F}_4\mid 2 \tilde{F}_1=F,\dots, 2 \tilde{F}_4=F
\rangle\cong\ZZ^2\times(\ZZ/2)^3
\]
where $F$ is a generic fiber.
We want to study the connected double coverings where $D=0$; as for $H$ we can choose any $2$-torsion divisor.
We can consider two types of such divisors, namely $H_1=\tilde{F}_1-\tilde{F}_2$ or $H_2=\tilde{F}_1-\tilde{F}_2+\tilde{F}_3-\tilde{F}_4$. For $H_j$ the divisor $L_j^{(1)}=H_j$, i.e., if we want to compute the first Betti number
of those coverings, then we need $H^1(S,\cO_S(L_j^{(1)}))$. Note that these two divisors are numerically trivial and
the values of $\Delta$-invariant are $0$ for the singular points not in the support of $H_j$ and
$\frac{1}{4}$ for the other ones. Hence
\[
\chi(S,\cO_S(H_1))=1-4\frac{1}{4}=0,\qquad \chi(S,\cO_S(H_2))=1-8\frac{1}{4}=-1.
\]
Using Serre duality and the ideas of the previous sections, we note that $H^0(S,\cO_S(H_j))=H^2(S,\cO_S(H_j))=0$.
Hence, $h^1(S,\cO_S(H_1))=0$ and $h^1(S,\cO_S(H_2))=1$. Hence the double covering associated with the divisor $H_1$ has
first Betti number~$0$, as expected since it is again a ruled surface with 8 singular double points as~$S$.
The first Betti number of the double covering associated with the divisor $H_2$ is~$2$; in fact it is the product
of an elliptic curve and $\bp^1$.
\end{ejm}

\bibliographystyle{plain}

\end{document}